\documentclass[a4paper,11pt]{article}
\usepackage[top=3.0cm, bottom=3.0cm, inner=3.0cm, outer=3.0cm, includefoot]{geometry}

\usepackage[utf8]{inputenc}
\usepackage[T1]{fontenc}
\usepackage{authblk}

\usepackage{verbatim}
\usepackage{multirow}
\usepackage{multicol}

\usepackage{amssymb}
\usepackage{amsmath}
\usepackage{mathtools,bbm}
\usepackage{graphicx,tikz}
\usepackage{amsthm}
\usepackage{caption,subcaption}

\usepackage{color}
\usepackage{enumerate}

\usepackage{cancel}

\newcommand{\Hess}{\operatorname{Hess}}
\newcommand{\Ric}{\operatorname{Ric}}
\newcommand{\grad}{\operatorname{grad}}
\newcommand{\trace}{\operatorname{tr}}
\newcommand{\IR}{{\mathbb{R}}}

\newcommand{\IN}{{\mathbb{N}}}

\newcommand{\Q}{{\mathcal{Q}}}
\newcommand{\K}{{\mathcal{K}}}

\newcommand{\N}{{\mathcal{N}}}
\newcommand{\A}{\mathcal{A}}
\newcommand{\B}{\mathcal{B}}

\renewcommand*{\v}{\mathbf{v}}

\newcommand{\w}{\mathbf{w}}
\newcommand{\x}{\mathbf{x}}

\newcommand{\eChar}{\begin{enumerate}[(i)]}
\newcommand{\eCharR}{\begin{enumerate}[(a)]}
\newcommand{\eBr}{\begin{enumerate}[(1)]}

\newcommand{\vol}{\operatorname{vol}}

\newcommand{\diag}{\operatorname{diag}}

\DeclareMathOperator*{\argmax}{arg\,max}

\newcommand{\Abstract}

\title
{
Bakry-\'Emery curvature on graphs as an eigenvalue problem
}

\author[1]{David Cushing}
\author[2]{Supanat Kamtue}
\author[3]{Shiping Liu}
\author[2]{Norbert Peyerimhoff}
\affil[1]{School of Mathematics, Statistics and Physics, Newcastle University, Newcastle upon Tyne}
\affil[2]{Department of Mathematical Sciences, Durham University, Durham}
\affil[3]{School of Mathematical Sciences and CAS Wu Wen-Tsun Key Laboratory of Mathematics, University of Science and Technology of China, Hefei} %, Anhui, 230026, China}
\date{\today}

\date{\today}

\theoremstyle{plain}
\newtheorem{lemma}{Lemma}[section]
\newtheorem{theorem}[lemma]{Theorem}
\newtheorem{proposition}[lemma]{Proposition}
\newtheorem{corollary}[lemma]{Corollary}

\theoremstyle{definition}

\newtheorem{definition}[lemma]{Definition}
\newtheorem{remark}[lemma]{Remark}

\newtheorem{example}[lemma]{Example}

\numberwithin{equation}{section}

\begin{document}

\maketitle

\pagestyle{plain}

\begin{abstract}
 In this paper, we reformulate the Bakry-\'Emery curvature on a weighted graph in terms of the smallest eigenvalue of a rank one perturbation of the so-called curvature matrix using Schur complement. This new viewpoint allows us to show various curvature function properties in a very conceptual way. We show that the curvature, as a function of the dimension parameter, is analytic, strictly monotone increasing and strictly concave until a certain threshold after which the function is constant. Furthermore, we derive the curvature of the Cartesian product using the crucial observation that the curvature matrix of the product is the direct sum of each component. Our approach of the curvature functions of graphs can be employed to establish analogous results for the curvature functions of weighted Riemannian manifolds. Moreover, as an application, we confirm a conjecture (in a general weighted case) of the fact that the curvature does not decrease under certain graph modifications.
%This allows us to derive an analogous result about the Cartesian product of weighted Riemannian manifolds.
\end{abstract}

%%%%%%%%%%%%%%%%%%%%%%%%%%%%%%%%%%%%%%%%%%%%%%%%%%%%%
\section{Introduction and statements of result}
%%%%%%%%%%%%%%%%%%%%%%%%%%%%%%%%%%%%%%%%%%%%%%%%%%%%%

The Ricci curvature is a fundamental notion in Riemannian geometry. It is also an essential ingredient in Einstein's formulation of general relativity. Lower Ricci curvature bounds on a Riemannian manifold allow one to extract various global geometric and topological information \cite{Jost,Petersen}. The notion of Ricci curvature or its lower bound has been extended in various ways to general metric measure spaces. One of such extensions is Bakry-\'Emery's curvature dimension inequalities $CD(\K, N)$ \cite{Bakry94,BE85}. Bakry and \'Emery demonstrated that lower Ricci curvature bounds can be understood entirely in terms of the Laplace-Beltrami operator: On an n-dimensional Riemannain manifold $(M^n,g)$, for any $N\in [n,\infty]$, the Ricci curvature is lower bounded by $\K$ at a point $x\in M$ if and only if the inequality $CD(\K,N)$, which can be formulated purely in terms of the Laplace-Beltrami operator, holds at $x$ \cite[pp.93-94]{Bakry94}.

Bakry-\'Emery theory has been a source of spectacular mathematical results \cite{BGL}. In recent years, the discrete Bakry-\'Emery theory on graphs has become an active emerging research field. There are a growing number of articles investigating this theory, see e.g., \cite{CLY14,CKLLS20,CKKLP20,CKPW20,CLMP20,CLP20,Elworthy91,FS18,GL17,Hua19,HL17,HL19,HM17,HMW19,JL14,KM18,KMY17,KKRT16,LM18,LL15,LY10,LMP18,LMP17,LMPR19,LP18,L19,MST20,Mun17,Mun19,MR20,Salez211,Salez212,Schmuckenschlager98,SWZ19}.
Let us mention here important related works on non-linear discrete curvature dimension inequalities, see e.g., \cite{BHLLMY13,DKZ17,GLLY19,HLLY19,Mun18}.

A basic fact about the \emph{optimal} lower Ricci curvature bound at a point $x$ of a Riemannian manifold $(M^n, g)$ is that it is equal to the smallest eigenvalue of the Ricci curvature tensor at $x$ (when treated as a symmetric $(1,1)$-tensor) \cite[Section 3.14]{Petersen}.

In this paper, we provide an analogue of this basic fact in discrete Bakry-\'Emery theory. That is, we reformulate the \emph{optimal} lower curvature bound $\K$ in Bakry-\'Emery's curvature dimension inequality $CD(\K,N)$ at a vertex $x$ of a weighted graph as the smallest eigenvalue of a rank-one perturbation of the so-called \emph{curvature matrix} (see Theorem \ref{thm:eigenvalue_main}). The curvature matrix at $x$ is of size $m\times m$, where $m$ is the number of neighbours of $x$ in the graph. This might be surprising at first glance: Graphs are discrete and there are no way to define any curvature tensor directly. For instance, there are even no chain rule and the Laplacian is not a diffusion operator \cite{LY10,BHLLMY13}. We achieve our result and conceive the concept of curvature matrix by combing an idea of Schmuckenschl\"ager \cite{Schmuckenschlager98} with the trick of Schur complements \cite{BV04,Car86,Gal19}, see also \cite[Proposition 5.13]{CLP20}. This new viewpoint leads to a confirmation of \cite[Conjecture 6.13]{CLP20} concerning the monotonicity of the Bakry-\'Emery curvature under certain graph modifications.

We further study the Bakry-\'Emery curvature as a function of the dimension parameter. Building upon the new viewpoint, we study the shape of the Bakry-\'Emery curvature functions systematically, especially the relation between the shape of the function and the spectrum of the curvature matrix. Very interestingly, the curvature matrix of a Cartesian product of two graphs is simply the direct sum of the curvature matrix of each graph.
 We use this to prove that the curvature function of Cartesian product is the star product (see Definition \ref{def:star_product}) of the curvature function of each factor.

The method we developed is also applicable to the setting of weighted Riemannian manifolds. Our results about the curvature functions of graphs can be transferred to the weighted manifold setting straightforwardly. In particular, we derive an analogous result about the curvature functions of Cartesian products of weighted Riemannian manifolds.

In the sequel, we will survey our results in more detail.

Let $G=(V,w,\mu)$ be a weighted graph consisting of a vertex set $V$, a vertex measure $\mu:V \to \IR^{+}$, and an edge-weight function $w:V\times V \to \IR^{+}\cup\{0\}$ which is a symmetric function with $w_{xx}=0$ for all $x\in V$. Two vertices $x,y\in V$ are adjacent if and only if $w_{xy}>0$. The graph $G$ is assumed to be \emph{locally finite}, that is, each vertex has only finitely many neighbours. %The \emph{combinatorial distance function} $\dist:V\times V \to \IN \cup\{0\}$ is the length of shortest path between two vertices on the underlying graph of $G$.
For $r\in \IN$, the \emph{combinatorial sphere} (resp. \emph{ball}) of radius $r$ centered at $x\in V$, denoted by $S_r(x)$ (resp. $B_r(x)$), is the set of all vertices whose minimum number of edges from $x$ is equal to (resp. less than or equal to) $r$. In particular, $S_1(x)$ contains all neighbours of $x$.

Furthermore, let $d_x := \sum_{y\in V} w_{xy}$ be the \emph{vertex degree} of $x$, and $p_{xy}:=\frac{w_{xy}}{\mu_x}$ be the \emph{transition rate} from $x$ to $y$. In the special case of $d_x=\mu_x$ (that is, $ \sum_{y\in V}p_{xy}=1$) for all $x\in V$, the terms $p_{xy}$ can be understood as transition probabilities of a reversible Markov chain.
Another special situation is a \emph{non-weighted} (or combinatorial) graph $G=(V,E)$ where $E$ is the set of edges (without loops and multiple edges), that is, $\mu\equiv 1$ and $w_{xy}=1$ iff $x$ is adjacent to $y$, and $w_{xy}=0$ otherwise.

The Laplacian $\Delta:C(V) \to C(V)$ (where $C(V)$ is the vector space of all functions $f:V\to \IR$) is given by
\[\Delta f(x):=\frac{1}{\mu_x}\sum_{y\in V}w_{xy}(f(y)-f(x))=\sum_{y\in V}p_{xy}(f(y)-f(x)).\]
The Laplacian associated to non-weighted graphs is also known as the \emph{non-normalised Laplacian}.

The Laplacian $\Delta$ gives rise to the symmetric bilinear forms $\Gamma$ and $\Gamma_2$, namely,
\begin{align*}
2\Gamma(f,g) &:= \Delta(fg) - f\Delta g - g\Delta f,\\
2\Gamma_2(f,g) &:= \Delta(\Gamma(f,g)) - \Gamma(f,\Delta g) - \Gamma(g,\Delta f),
\end{align*}
with additional notations $\Gamma (f):=\Gamma (f,f)$ and $\Gamma_2(f):=\Gamma_2(f,f)$.

These bilinear forms are important for the following Ricci curvature notion due to Bakry-\'Emery \cite{BE85}, which is motivated by a fundamental identity in Riemannian Geometry called Bochner's formula.

\begin{definition}[Bakry-\'Emery curvature]\label{defn:BEcurvature} Let $G=(V,w,\mu)$ be a locally finite weighted graph. Let $\K\in \mathbb{R}$ and $N\in (0,\infty]$. We say that a vertex $x\in V$ satisfies the Bakry-\'Emery's \emph{curvature-dimension inequality} $CD(\K,N)$, if for any $f:V\to \mathbb{R}$, we have
\begin{equation}\label{eq:CDineq}
	\Gamma_2(f)(x)\geq \frac{1}{N}(\Delta f(x))^2+\K\Gamma(f)(x),
\end{equation}
where $N$ is a dimension parameter and $\K$ is regarded as a lower Ricci curvature bound at $x$. The Bakry-\'Emery curvature, denoted by $\K(G,x;N)$, is then defined to be the largest $\K$ such that $x$ satisfies $CD(\K,N)$.
\end{definition}

The \emph{Bakry-\'Emery curvature function} of $x$, namely $\K_{G,x}(N):=\K(G,x;N)$ can be reformulated as the solution to the following semidefinite programming:
\begin{align} \label{eq:SDP_defn}
	&\text{maximize}\,\,\, K   \tag{$P$} \\
	&\text{subject to}\,\,\,\Gamma_2(x)-\frac{1}{N}\Delta(x)\Delta(x)^\top - K\Gamma(x) \succeq 0, \nonumber
\end{align}
where the symmetric matrices $\Gamma(x)$ and $\Gamma_2(x)$ correspond to the symmetric bilinear forms $\Gamma$ and $\Gamma_2$ at $x$. The explicit expression of these matrices is given in Appendix \ref{sect:appendix}. Here, $M\succeq 0$ (resp. $M\succ 0$) means $M$ is positive semidefinite (resp. positive definite). The above computing method has been studied by Schmuckenschl\"ager \cite{Schmuckenschlager98}, and later on in \cite{LP18}, \cite{LMP19} and \cite{CLP20}.

In this paper, we reformulate the above semidefinite programming problem as a smallest eigenvalue problem by employing the Schur complement of a square block matrix $M_{22}$ in $M= \begin{pmatrix}
		M_{11} & M_{12} \\
		M_{21} & M_{22}
	\end{pmatrix}$, namely $M/M_{22}:=M_{11}-M_{12}M_{22}^{-1}M_{21}$, applied to the matrix
\begin{equation*}
	\Gamma_2(x)_{\hat{1}}=\left(
	\begin{array}{cc}
		\Gamma_2(x)_{S_1,S_1} & \Gamma_2(x)_{S_1,S_2} \\
		\Gamma_2(x)_{S_2,S_1} & \Gamma_2(x)_{S_2,S_2} \\
	\end{array}
	\right).
\end{equation*}
Here the matrix $\Gamma_2(x)_{\hat{1}}$ refers to the principle submatrix of $\Gamma_2(x)$ obtained by removing its first row and column corresponding to the central vertex $x$. The matrix $\Gamma_2(x)_{S_i,S_j}$ refers to the submatrix of $\Gamma_2(x)$ whose rows and columns are indexed by the vertices of the combinatorial spheres $S_i(x)$ and $S_j(x)$.

We use the notation $Q(x):=\Gamma_2(x)_{\hat{1}} / \Gamma_2(x)_{S_2,S_2}$ for simplicity, and define
\begin{align} \label{eq:A_defn}
A_\infty(x) &:= 2\diag(\mathbf{v}_0(x))^{-1} Q(x) \diag(\mathbf{v}_0(x))^{-1},  \nonumber \\
A_N(x) &:= A_\infty(x) - \frac{2}{N}\mathbf{v}_0(x)\mathbf{v}_0(x)^\top,
\end{align}
where $\mathbf{v}_0(x) := (\sqrt{p_{xy_1}} \ \sqrt{p_{xy_2}} \ ... \ \sqrt{p_{xy_m}})^\top$ with $S_1(x)=\{y_1,y_2,...,y_m\}$ labelling the neighbours of $x$. Note that the matrices $Q(x), A_\infty(x), A_N(x)$ are all symmetric matrices, and that $A_N(x)$ is a rank one perturbation of $A_\infty(x)$. All our subsequent results are based on the following theorem.
\begin{theorem} \label{thm:eigenvalue_main}
Let $G=(V,w,\mu)$ be a weighted graph. For $x\in V$ and $N\in (0,\infty]$,
the Bakry-\'Emery curvature $\K_{G,x}(N)$ is the smallest eigenvalue of the symmetric matrix $A_N(x)$, that is,
\begin{equation*}
	\K_{G,x}(N) =  \lambda_{\min}(A_N(x)).
\end{equation*}
\end{theorem}

Theorem \ref{thm:eigenvalue_main} is proved in Section \ref{sect:curvature_reformulation}. This concept of the curvature expression as eigenvalues was discussed in \cite[Section 5]{CLP20} in the special case that a vertex $x$ is $S_1$-out regular. Henceforth we will use the simplified notations $\mathbf{v}_0$, $Q$, $A_\infty$ and $A_N$ for the vector $\mathbf{v}_0(x)$ and the matrices $Q(x)$, $A_\infty(x)$ and $A_N(x)$, where $x$ is a fixed vertex of $G$. We may refer to the matrix $A_\infty=A_\infty(x)$ as \textbf{\emph{the curvature matrix}} of $x$.

The relation $\K_{G,x}(N) =  \lambda_{\min}(A_N)$ allows us to investigate various properties of the curvature function $\K_{G,x}: (0,\infty] \to \IR$. Some of the results here were already introduced in \cite{CLP20} in the case of non-weighted graphs, but this paper presents a unified and simplified approach to these results by employing the variational description of minimal eigenvalues via the Rayleigh quotient
\[ \lambda_{\min{}} (A_N) = \inf_{v\not=0} \frac{v^\top A_N v}{v^\top v}.\]

We first describe the shape of the curvature functions (see proofs in Section \ref{sect:curv_fct_properties}).
\begin{theorem} \label{thm:continuous_and_threshold}
Let $G = (V,w,\mu)$ be a weighted graph, and fix $x \in V$. Then the curvature function $\K_{G,x}: (0,\infty] \to \IR$ is continuous and there exists a unique threshold $N_1 \in (0,\infty]$ (possibly, $N_1 = \infty$) with the following properties:
\begin{itemize}
\item[(i)] $\K_{G,x}$ is analytic, strictly monotone increasing and strictly concave on $(0,N_1]$ with \\ $\lim_{N \to 0} \K_{G,x}(N) = -\infty$ and $\lim_{N \to N_1} \K_{G,x}(N) =: K_1 <  \infty$.
\item[(ii)] $\K_{G,x}$ is constant on $[N_1,\infty]$ and equal to $K_1$.
\end{itemize}
\end{theorem}
In fact, the threshold $N_1$ is the minimal $N\in (0,\infty]$ for which $\lambda_{\min}(A_N)$ is not simple. Another interesting threshold is given when the curvature function vanishes. Here we have the following result.

\begin{proposition} \label{cor: N_threshold}
Assume that $A_\infty \succ 0$ (that is $\K_{G,x}(\infty)>0$). Then there exists a unique $N_0\in (0,\infty)$ such that $\K_{G,x}(N_0)=0$, and it is given by
\[ N_0= 2\mathbf{v}_0^\top A_\infty^{-1}\mathbf{v}_0 = 2\sum_{i,j} \sqrt{p_{xy_i}p_{xy_i}} (A_\infty^{-1})_{ij}. \]
\end{proposition}

Next we prove in Section \ref{sect:bounds} the following curvature bounds. The upper bound, in particular, plays an important role in our curvature analysis, where we study the situation when this upper bound is attained (called curvature sharpness; see the definition below). The notion of curvature sharpness was introduced \cite{CLP20} and studied in, e.g., \cite{CKPW20}.
\begin{theorem}[Upper and lower curvature bounds] \label{thm:lower_upper_bound}
	Let $G=(V,w,\mu)$ be a weighted graph. Then we have for $x\in V$ and $N\in (0,\infty]$,
	\begin{equation} \label{eq:lower_upper_bound}
		\K_{G,x}(\infty) - \frac{2}{N}\frac{d_x}{\mu_x} \le \K_{G,x}(N) \stackrel{(^*)}{\le} \K^{0}_{\infty}(x) - \frac{2}{N}\frac{d_x}{\mu_x}
	\end{equation}
	with \[\K^{0}_{\infty}(x) := \frac{\v_0^\top A_\infty \v_0}{\v_0^\top \v_0} = \frac{1}{2}\left( \frac{d_x}{\mu_x}+ 3\frac{\mu_x}{d_x} p_{xx}^{(2)} - \frac{\mu_x}{d_x} \sum_{z\in S_2(x)} p_{xz}^{(2)}\right). \] Here
	we use the notation $p_{uv}^{(2)}:=\sum_{w\in V} p_{uw}p_{wv}$. Moreover, a vertex $x\in V$ is called \textbf{\emph{$N$-curvature sharp}} iff $(^*)$ in \eqref{eq:lower_upper_bound} holds with equality.
\end{theorem}

The next proposition clarifies the relation between curvature sharpness and the appearance of the following shapes of the curvature function $\K_{G,x}$:
\begin{itemize}
\item $\K_{G,x}(N)= c - \frac{2}{N}\frac{d_x}{\mu_x}$ (with a constant $c\in \IR$) for all $N$ near $0$, and
\item $\K_{G,x}(N)$ is constant for $N$ near $\infty$.
\end{itemize}

\begin{proposition} \label{prop:curv_sharp_shape}
If $x$ is $N_1$-curvature sharp for some $N_1\in (0,\infty]$, it is also $N$-curvature sharp for all $N \in (0,N_1]$. If $x$ is $N_1$-curvature sharp for a maximally chosen $N_1$, then this $N_1$ is the threshold mentioned in Theorem \ref{thm:continuous_and_threshold}, and hence $\K_{G,x}(N)=\K^{0}_{\infty}(x) - \frac{2}{N}\frac{d_x}{\mu_x}$ for all $N\in (0,N_1]$ and $\K_{G,x}$ is constant on $[N_1,\infty]$. Conversely, if $\K_{G,x}(N) = c - \frac{2}{N}\frac{d_x}{\mu_x}$ for some constant $c\in \IR$ on some nontrivial interval $(N',N'')$, then $x$ is $N''$-curvature sharp.
\end{proposition}

The following proposition provides insights into relations between curvature sharpness and the spectrum of the curvature matrix $A_\infty$.
\begin{proposition} \label{prop:v0_eigenvalue}
Let $G=(V,w,\mu)$ be a weighted graph and fix a vertex $x\in V$. Denote $E_{\min{}}(A_\infty)$ to be the minimal eigenspace of $A_\infty$.
\begin{itemize}
\item[(i)] $\mathbf{v}_0$ is an eigenvector of $A_\infty$ if and only if $x$ is $N_1$-curvature sharp for some $N_1\in (0,\infty]$.
\item[(ii)] $\mathbf{v}_0 \in E_{\min{}}(A_\infty)$ if and only if $x$ is $\infty$-curvature sharp.
\item[(iii)] $\mathbf{v}_0$ is perpendicular to $E_{\min{}}(A_\infty)$ if and only if $\K_{G,x}$ is constant on $[N_1, \infty]$ for some $N_1<\infty$.
\end{itemize}
\end{proposition}
The proofs of Propositions \ref{prop:curv_sharp_shape} and \ref{prop:v0_eigenvalue} are provided in Section \ref{sect:spectral_relation}.

\begin{remark} \label{rem:example_not curvature_sharp}
If $\mathbf{v}_0$ is an eigenvector of $A_\infty$ corresponding to a non-smallest eigenvalue of $A_\infty$, then $\mathbf{v}_0$ is perpendicular to $E_{\min{}}(A_\infty) $. The converse is not true; a counterexample is the non-weighted Cartesian product $P_3\times P_2$, discussed in Example \ref{ex:Emin_perp_v0}. In this example,  $\mathbf{v}_0$ is perpendicular to $E_{\min{}}(A_\infty) $ but it is not an eigenvector of $A_\infty$, and its curvature function $\K_{G,x}$ is strictly increasing and strictly concave (but not curvature sharp) on $(0,N_1]$ and constant on $[N_1, \infty]$.
\end{remark}

In Section \ref{sect:Cartesian}, we discuss an important property of the curvature matrix $A_\infty$, that is, the curvature matrix of the the Cartesian product of two graphs is simply the direct sum of the curvature matrices of each graph.

\begin{definition}[weighted Cartesian product]
Given two weighted graphs $G,G'$ and two fixed positive numbers $\alpha,\beta\in \IR^{+}$, the weighted Cartesian product $G\times_{\alpha,\beta} G'$ is defined with the following weight function and vertex measure: for $x,y\in G$ and $x',y'\in G'$,
\begin{align*}
w_{(x,x')(y,x')} &:= \alpha w_{xy} \mu_{x'},\\
w_{(x,x')(x,y')} &:= \beta w_{x'y'} \mu_{x},\\
\mu_{(x,x')} &:=\mu_x \mu_{x'}.
\end{align*}
\end{definition}

The parameters $\alpha$ and $\beta$ serve two purposes.
\begin{enumerate}
\item In the case of non-weighted graphs $G$ and $G'$ (i.e.,  $\mu \equiv 1$ and $w\in \{0,1\}$), the choice of $\alpha=\beta=1$ gives the usual Cartesian product graph $G\times G'$.
\item In the case of $G$ and $G'$ representing Markov chains (i.e., when $\sum_{y} w_{xy}=\mu_x$ and $\sum_{y'} w_{x'y'}=\mu_{x'}$), the choice of $\alpha+\beta=1$ gives the weighted product $G\times_{\alpha,\beta} G$ which represents the random walk with probability $\alpha$ and $\beta$ following horizontal and vertical edges, respectively.
\end{enumerate}

\begin{theorem} \label{thm:cartesian_curv_matrix}
The curvature matrix of the product $G\times_{\alpha,\beta} G'$ is the weighted direct sum of the curvature matrices $G$ and $G'$:
\begin{align*}
A^{G\times_{\alpha,\beta} G'}_\infty((x,x'))= \alpha A^{G}_\infty(x) \oplus \beta A^{G'}_\infty(x').
\end{align*}
\end{theorem}

As a consequence, we give a new proof (in a more general case of weighted graphs) of the fact that the curvature function of a Cartesian product is the star product of the curvature function in each factor (see Theorem \ref{thm:cartesian_curv_fct} below).

\begin{definition}[star product {\cite[Definition 7.1]{CLP20}}] \label{def:star_product}
Let $f_1,f_2: (0,\infty] \to \IR$ be continuous and monotone increasing functions with $\lim_{t\to 0} f_1(t)=\lim_{t\to 0} f_2(t)=-\infty$.
Then the function $f_1\ast f_2: (0,\infty] \to \IR$ is defined by
\begin{equation*}
f_1 \ast f_2 (t) := f_1(t_1) = f_2(t_2),
\end{equation*}
where $t_1+t_2=t$ such that $f_1(t_1)=f_2(t_2)$.
\end{definition}
Let us remark also that the star product is commutative and associative \cite[Propositions 7.5 and 7.6]{CLP20}.

\begin{theorem} \label{thm:cartesian_curv_fct}
The curvature function of the product $G\times_{\alpha,\beta} G'$ satisfies the following inequalities:
\begin{equation*}
\min \{ \alpha\K_{G,x}, \beta\K_{G',x'}\} \le  \K_{G\times_{\alpha,\beta} G',(x,x')}\le \max \{ \alpha\K_{G,x}, \beta\K_{G',x'}\}.
\end{equation*}
Consequently, we have $\K_{G\times_{\alpha,\beta} G',(x,x')} = (\alpha\K_{G,x}) \ast (\beta\K_{G',x'})$.
\end{theorem}

In Section \ref{sect:weighted_manifolds}, we discuss analogous results in the smooth setting of weighted manifolds. Consider a weighted Riemannian manifold $(M^n, g, e^{-V} d{\vol}_g)$ of dimension $n$, with a metric $g$, the volume element $d{\vol}_g$, and a smooth real function $V:M\to \IR$.
The Bakry-\'Emery curvature function $\K_{M,V,x}: (0,\infty]\to \mathbb{R}$ at $x\in M$ is defined as
%We define the Ricci curvature lower bound at $x\in M$ to be
\[ \K_{M,V,x}(N):= \inf_{v\in S_x(M)} \Ric_{N+n,V}(v,v),  \qquad \forall N\in (0,\infty],\]
where $S_x(M)$ is the space of unit tangent vectors at $x$, and \[\Ric_{N,V} := \Ric + \Hess V - \frac{\grad V \otimes \grad V}{N-n},\,\,\forall N\in (n, \infty], \] where we follow the notation in \cite[Equation (14.36)]{Villani}.
We define $\K_{M,V,x}$ as a function on the interval $(0,\infty]$ instead of on $(n,\infty]$ to make it compatible with the curvature functions of graphs.
All the results (Theorems \ref{thm:continuous_and_threshold}, \ref{thm:lower_upper_bound} and \ref{thm:cartesian_curv_fct}, Propositions \ref{cor: N_threshold}, \ref{prop:curv_sharp_shape} and \ref{prop:v0_eigenvalue}) have analogous counterparts in the manifold case. For example, the upper bound is $\K_{M,V,x}(N) \le \K^0_\infty(x) - \frac{1}{N}\|\grad V\|^2$ with
\begin{equation*}
\K^0_\infty(x)=\Ric_x\left(\frac{\grad V}{\|\grad V\|}, \frac{\grad V}{\|\grad V\|}\right)+\frac{\grad V(x)}{\|\grad V(x)\|}\left(\|\grad V\|\right).
\end{equation*}
We also show that the Cartesian product of two weighted manifolds $(M_i^{n_i}, g_i, e^{-V_i} d{\vol}_{g_i})$, $i\in \{1,2\}$ has the Bakry-\'Emery curvature function
\begin{equation}
\K_{M_1\times M_2 , V_1\oplus V_2, (x_1,x_2)} = \K_{M_1,V_1,x_1} \ast \K_{M_2,V_2,x_2}.
\end{equation}
Furthermore, we may define the \emph{generalised scalar curvature} for a weighted Riemannian manifold $(M,g, e^{-V}d{\vol})$ to be the trace of the Ricci tensor
\begin{equation} \label{eq:scalar_curv_defn}
S_{M,V,x}(N) := \trace \Ric_{N+n,V},\,\,\forall\,\,N\in (0,\infty].
\end{equation}
In Example \ref{ex:2-sphere}, we investigate curvature sharpness properties of weighted $2$-spheres and derive explicit formulas for the curvatures $\K_{M,V,x}$ and $S_{M,V,x}$.
At the end of Section \ref{sect:weighted_manifolds}, we also discuss an interesting connection between curvature sharpness and Ricci solitons (see Theorem \ref{thm:soliton}).

In Section \ref{sect:geom_property}, we prove the following curvature results related to the geometric structure of $B_2(x)$.  First, we define for a graph $G$ an analogue to the generalised scalar curvature, namely
$$S_{G,x}(N):= \trace A_N, \,\,\forall\,\,N\in (0,\infty].$$ In contrast to Ricci curvature, this scalar curvature can be formulated explicitly for non-weighted graphs
in terms of the vertex degrees, the number of triangles and the size of $S_2(x)$.

Let us denote $S_1(x)=\{y_1,\ldots,y_{d_x}\}$. At a vertex $x$ in a non-weighted graph, we define the out-degree $d^+_{y_i}$ of $y_i\in S_1(x)$  to be the number of neighbours of $y_i$ in $S_2(x)$ and the in-degree $d^-_{z}$ of $z\in S_2(x)$ to be the number of neighbors of $z$ in $S_1(x)$.
\begin{proposition} \label{prop:scalar_curv}
Let $G=(V,w,\nu)$ be a non-weighted graph. Then
\begin{itemize}
  \item [(i)] The curvature matrix at a vertex $x\in V$ is given by
  \begin{equation}\label{eq:CurMatrixNonweighted}
  A_\infty(x)=-2\Delta_{S_1(x)}-2\Delta_{S_1'(x)}+J+\frac{3-d_x}{2}\mathrm{Id}-\frac{1}{2}\diag((d_{y_1}^+,\ldots, d_{y_{d_x}}^+)^\top),
  \end{equation}
  where $J$ is the $d_x\times d_x$ all-one matrix, $\Delta_{S_1(x)}$ is the Laplacian matrix of the subgraph of $G$ induced by $S_1(x)$ and $\Delta_{S_1'(x)}$ is the Laplacian matrix of the weighted graph with vertex set $S_1(x)$, vertex measure $\mu\equiv 1$, and edge weights $w_{y_iy_j}^{S_1'(x)}=\sum_{z\in S_2(x)}\frac{w_{y_iz}w_{y_jz}}{d^-_{z}}$ for $i\neq j$ and $0$ otherwise.
  %\[w_{y_iy_j}^{S_1'(x)}=\sum_{z\in S_2(x)}\frac{w_{y_iz}w_{y_jz}}{d^-_{z}} \,\,\text{for }\,\, i\neq j,\,\,\text{and}\,\,0\,\,\text{otherwise}.\]
  \item [(ii)]The generalised scalar curvature at a vertex $x\in V$ is given by
\begin{equation} \label{eq:scalar_compute_nonweight}
S_{G,x}(N) = d_x-\frac{d_x^2}{2} + \frac{3}{2}\sum_{y\in S_1(x)} d_y + \sharp_{\triangle}(x) - 2|S_2(x)| - \frac{2}{N}d_x,
\end{equation}
where $\sharp_{\triangle}(x)$ denotes the number of triangles ($3$-cycles) containing the vertex $x$. In particular, for a $d$-regular tree, we have $S_{G,x}(N)=d(3-d) - \frac{2d}{N}$.
\end{itemize}
\end{proposition}
 It follows from \eqref{eq:scalar_compute_nonweight} that the scalar curvature is larger in the presence of more triangles or a smaller two-sphere. Secondly, we provide a sufficient criterion for curvature sharpness.
\begin{theorem} \label{thm:curv_sharp_criterion}
Let $G=(V,w,\mu)$ be a weighted graph. A vertex $x\in V$ is $N$-curvature sharp for some $N\in(0,\infty]$ if the following two homogeneity properties of $x$ are satisfied:
\begin{itemize}
	\item $x$ is $S_1$-in regular: $p^{-}(y)=p_{yx}$ is independent of $y\in S_1(x)$,
	\item $x$ is $S_1$-out regular: $p^{+}(y)=\sum_{z\in S_2(x)} p_{yz}$ is independent of $y\in S_1(x)$.
\end{itemize}
\end{theorem}
In the case of the non-weighted graphs, the $S_1$-in regularity is always satisfied ($p^{-}(y)=1$), and we even have equivalence between $S_1$-out regularity and $N$-curvature sharpness for some $N\in (0,\infty]$ \cite[Corollary 5.10]{CLP20}. In fact, one can check directly from \eqref{eq:CurMatrixNonweighted} the following fact in the case of non-weighted graphs: $\textbf{v}_0$ is an eigenvector of $A_\infty$ if and only if $x$ is $S_1$-out regular. Therefore, our Proposition \ref{prop:v0_eigenvalue}(i) is a substantial extension of \cite[Corollary 5.10]{CLP20} in the case of general weighted graphs..

 Our final result states that the curvature is nondecreasing under certain graph modifications.
\begin{theorem}\label{thm:graph_modify_conjecture}
Let $G=(V,w,\mu)$ be a weighted graph and fix a vertex $x\in V$. Assume that $x$ is $S_1$-in regular, i.e., $p^{-}(y)=p_{yx}$ is independent of $y\in S_1(x)$. Consider a modified weighted graph $\widetilde{G}$ obtained from $G$ by one of the following operations:
	\begin{itemize}
		\item[(O1)] Increase the edge-weight between a fixed pair $y,y'\in S_1(x)$ with $y\not=y'$ by $\tilde{w}_{yy'}=w_{yy'}+C_1$ for any constant $C_1>0$.
		\item[(O2)] Delete a vertex $z_0\in S_2(x)$ and remove all of its incident edges, i.e., $\tilde{w}_{yz_0}=0$ for all $y\in S_1(x)$. Increase the edge-weight between all pairs $y,y'\in S_1(x)$ with $y\not=y'$ by
		\begin{equation} \label{eq:edge_modify}
		\tilde{w}_{yy'}=w_{yy'}+C_2 w_{yz_0}w_{z_0y'}
		\end{equation}
		with any constant $\displaystyle C_2 \ge \frac{p^{-}(y)}{\mu_x p_{xz_0}^{(2)}}$.
	\end{itemize}
Then $\K_{\widetilde{G},x}(N) \ge \K_{G,x}(N)$ for any $N\in (0,\infty]$.
\end{theorem}
The part (O2) of the above theorem confirms Conjecture 6.13 in \cite{CLP20} in the case of non-weighted graphs where we consider $\tilde{w}_{yy'}=w_{yy'} + 1$ for all pairs $y,y'\in S_1(x)$ of neighbours of $z_0$. In this special case, the constant $C_2=1$ is bigger or equal to the threshold \[\frac{p^{-}(y)}{\mu_x p_{xz_0}^{(2)}}=\frac{1}{p_{xz_0}^{(2)}}=:\frac{1}{\text{in-degree of } z_0}.\]
In fact, the $S_1$-in regularity condition at $x$ can be weakened to $S_1$-in regularity at $x$ for the involved vertices in $S_1(x)$. In the operation (O1) we only require $p_{yx}=p_{y'x}$, and in (O2) we require $p_{yx}$ is constant for all $y\in S_1(x)$ such that $w_{yz_0}\not=0$.

\textbf{Note:} \emph{After the submission of our first arXiv version, we became aware of the work by Siconolfi \cite{Siconolfi-proceeding, Siconolfi} in which the $\infty$-Bakry-\'Emery curvature $\K_{G,x}(\infty)$ is also formulated as an eigenvalue problem in the special case of non-weighted graphs.}

%%%%%%%%%%%%%%%%%%%%%%%%%%%%%%%%%%%%%%%%%%%%%%%%%%%%%
\section{Curvature reformulation} \label{sect:curvature_reformulation}
%%%%%%%%%%%%%%%%%%%%%%%%%%%%%%%%%%%%%%%%%%%%%%%%%%%%%

In this section, we prove the eigenvalue reformulation of the curvature (Theorem \ref{thm:eigenvalue_main}).
%our main result (Theorem \ref{thm:eigenvalue_main}), namely the eigenvalue reformulation of the curvature.
Recall the optimization problem which formulates the Bakry-\'Emery curvature $\K_{G,x}(N)$,
\begin{align} \label{eq:opt_problem_old}
	&\text{maximize}\,\,\, K   \tag{$P$} \\
	&\text{subject to}\,\,\,\Gamma_2(x)-\frac{1}{N}\Delta(x)\Delta(x)^\top - K\Gamma(x) \succeq 0, \nonumber
\end{align}
This curvature is a local concept and uniquely determined by the structure of the two-ball $B_2(x)$. In particular, the symmetric matrix $\Gamma_2(x)$ is of size $|B_2(x)|$, and the symmetric matrices $\Delta(x)\Delta(x)^\top$ and $\Gamma(x)$ are of sizes $|B_1(x)|$ (and trivially extended by zeros to matrices of sizes $|B_2(x)|$); see Appendix \ref{sect:appendix} for details.

Schmuckenschl\"ager \cite{Schmuckenschlager98} observed that the size of these matrices can be reduced by one: since $\Gamma_2(f), \Gamma(f), \Delta f$ all vanish for constant functions $f$, the curvature-dimension inequality $CD(\K,N)$ remains valid after shifting $f$ by an additive constant. It is therefore sufficient to verify $(\ref{eq:CDineq})$ for all functions $f:V\to \mathbb{R}$ with $f(x)=0$. This observation allows us remove from these matrices the row and column corresponding to the vertex $x$, and we are able to reformulate the above problem \eqref{eq:opt_problem_old} as
\begin{align} \label{eq:opt_problem_new}
	&\text{maximize}\,\,\, K   \tag{$P'$} \\
	&\text{subject to}\,\,\,M_{K,N}(x):=\left(\Gamma_2(x)-\frac{1}{N}\Delta(x)\Delta(x)^\top-K\Gamma(x)\right)_{S_1\cup S_2, S_1\cup S_2} \succeq 0, \nonumber
\end{align}

Next we recall the concept of the Schur complement, which allows us to further reduce the size of the involved symmetric matrices in \eqref{eq:opt_problem_new}.
\begin{lemma}[Schur complement] \label{lem:Schur}
Consider a real symmetric matrix $M= \begin{pmatrix}
	M_{11} & M_{12} \\
	M_{21} & M_{22}
\end{pmatrix}$,
where $M_{11}$ and $M_{22}$ are square submatrices, and assume that $M_{22} \succ 0$. The Schur complement $M/M_{22}$ is defined as
\begin{equation}\label{eq:Schur_defn}
	M/M_{22} := M_{11}-M_{12}M_{22}^{-1}M_{21}.
\end{equation}
Then $M/M_{22} \succeq 0$ if and only if $M \succeq 0$.
\end{lemma}
The proof of this lemma can be found in, e.g., \cite[Proposition 2.1]{Gal19} or \cite[Proposition 5.13]{CLP20}. We aim to apply this lemma for the symmetric matrix $M_{K,N}(x)$ given in \eqref{eq:opt_problem_new}. Since $\Delta(x)$ and $\Gamma(x)$ have zero entries in the $S_2(x)$-structure, it means the matrix $M_{K,N}(x)$ has the following block structure:
\begin{align*}
	M_{K,N}(x) = \left(\begin{array}{cc}
		\Gamma_2(x)_{S_1,S_1}-\frac{1}{N}\Delta(x)_{S_1}\Delta(x)_{S_1}^\top-K\Gamma(x)_{S_1,S_1} & \Gamma_2(x)_{S_1,S_2} \\ %\hline
		\Gamma_2(x)_{S_2,S_1} & \Gamma_2(x)_{S_2,S_1}
	\end{array}
	\right).
\end{align*}
By folding $M_{K,N}(x)$ into the upper left block, the Schur complement is given by
\begin{align*}
	&M_{K,N}(x)/ \Gamma_2(x)_{S_2,S_2}  \\
	&= \Gamma_2(x)_{S_1,S_1}-\frac{1}{N}\Delta(x)_{S_1}\Delta(x)_{S_1}^\top-K\Gamma(x)_{S_1,S_1} - \Gamma_2(x)_{S_1,S_2}\Gamma_2(x)_{S_2,S_2}^{-1}\Gamma_2(x)_{S_2,S_1}\\
	&= Q(x) -\frac{1}{N}\Delta(x)_{S_1}\Delta(x)_{S_1}^\top-K\Gamma(x)_{S_1,S_1},
\end{align*}
where $Q(x):=\Gamma_2(x)_{\hat{1}}/ \Gamma_2(x)_{S_2,S_2}$ denotes the folding of $\Gamma_2(x)_{\hat{1}}=
\begin{pmatrix}
	\Gamma_2(x)_{S_1,S_1} & \Gamma_2(x)_{S_1,S_2} \\
	\Gamma_2(x)_{S_2,S_1} & \Gamma_2(x)_{S_2,S_2}
\end{pmatrix}$.

The importance of $\Gamma_2(x)_{S_1,S_1}$ for a lower curvature bound was already mentioned in Schmuckenschl\"ager \cite[pp.194-195]{Schmuckenschlager98} (where he used the notation $A_{II}$).

Lemma \ref{lem:Schur} implies that
\begin{equation} \label{eq:K=argmax_old}
\K_{G,x}(N) = \argmax_{K}\left\{ Q(x) -\frac{1}{N}\Delta(x)_{S_1}\Delta(x)_{S_1}^\top-K\Gamma(x)_{S_1,S_1} \succeq 0 \right\}.
\end{equation}

We recall from Appendix \ref{sect:appendix} that $\Gamma(x)_{S_1,S_1}=\frac{1}{2}\diag(\Delta(x)_{S_1})$ and $\Delta(x)_{S_1}=(p_{xy_1} \ p_{xy_2} \ ... \ p_{xy_m})^\top$, where $S_1(x)=\{y_1,y_2,...,y_m\}$.

Denote the vector $\mathbf{v}_0 := \mathbf{v}_0(x) = (\sqrt{p_{xy_1}} \ \sqrt{p_{xy_2}} \ ... \ \sqrt{p_{xy_m}})^\top$. The maximum argument in \eqref{eq:K=argmax_old} does not change under the multiplication by $\diag(\mathbf{v}_0)^{-1} \succ 0$ both from left and right sides, that is,
\begin{equation} \label{eq:K=argmax_new}
	\K_{G,x}(N) = \argmax_{K}\left\{ \diag(\mathbf{v}_0)^{-1}Q(x)\diag(\mathbf{v}_0)^{-1} -\frac{1}{N}\mathbf{v}_0\mathbf{v}_0^\top-\frac{K}{2}{\rm Id} \succeq 0 \right\}.
\end{equation}

In other words,
\begin{align*}
\K_{G,x}(N) &= \lambda_{\min{}}(2\diag(\mathbf{v}_0)^{-1}Q(x)\diag(\mathbf{v}_0)^{-1} -\frac{2}{N}\mathbf{v}_0\mathbf{v}_0^\top)\\
&=\lambda_{\min{}}(A_\infty-\frac{2}{N}\mathbf{v}_0\mathbf{v}_0^\top) = \lambda_{\min{}}(A_N),
\end{align*}
where $A_\infty=A_\infty(x)$ and $A_N=A_N(x)$ are defined in \eqref{eq:A_defn}, and $\lambda_{\min{}}(A_N)$ denotes the smallest eigenvalue of $A_N$. This finishes the proof of Theorem \ref{thm:eigenvalue_main}.
\begin{remark} \label{rem:cube_K33}
It follows from the Appendix \eqref{eq:Q_ondiag}-\eqref{eq:A_infty_entries} that the curvature matrix at a vertex $x$ is completely determined by the weighted structure of the \emph{incomplete two-ball} around $x$, namely $B_2^{\rm inc}(x)$, which is obtained from the induced subgraph of $B_2(x)$ by removing all edges connecting vertices within $S_2(x)$. It is interesting to note however that two graphs can share the same curvature matrix, even when they have non-isomorphic $B_2^{\rm inc}(x)$. For example, both graphs $G_1$ and $G_2$, whose $B_2^{\rm inc}(x)$ are as in Figure \ref{fig:example_same_A_inf}, have their curvature matrix at $x$ equal to
$A^{G_1}_N(x) = 2{\rm Id}_4-\frac{2}{N}J_4 = A^{G_2}_N(x)$.

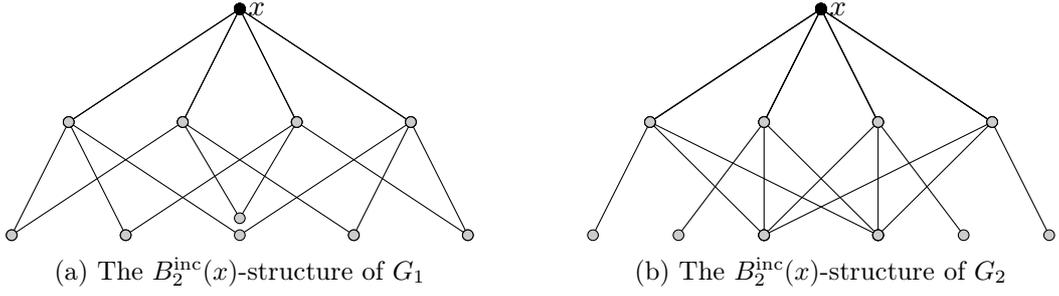
\begin{figure}[!htp]
\begin{center}
\tikzstyle{every node}=[circle, draw, fill=black!20, inner sep=0pt, minimum width=4pt]

    \begin{subfigure}[b]{0.45\textwidth}
        \centering
 \begin{tikzpicture}[scale=1.5]

 \foreach \x in {-.5,.5}
	{\draw (1.5,2) node{} -- (0,1) node{} -- (\x,0) node{};}
\foreach \x in {-.5,2.5}
	{\draw (1.5,2) node{} -- (1,1) node{} -- (\x,0) node{};}
\foreach \x in {.5,3.5}
	{\draw (1.5,2) node{} -- (2,1) node{} -- (\x,0) node{};}
\foreach \x in {2.5,3.5}
	{\draw (1.5,2) node{} -- (3,1) node{} -- (\x,0) node{};}
\draw (1,1) node{} -- (1.5,.15) node{} -- (2,1) node{};	
\draw (0,1) node{} -- (1.5,0) node{} -- (3,1) node{};	
\node at (1.5,2) [label=right: ${x}$, fill=black] {};
 	
\end{tikzpicture}

\caption{The $B_2^{\rm inc}(x)$-structure of $G_1$}
\label{fig:G_1}
    \end{subfigure}
\qquad
    \begin{subfigure}[b]{0.45\textwidth}
        \centering
 \begin{tikzpicture}[scale=1.5]

 \foreach \x in {-.5,1,2}
	{\draw (1.5,2) node{} -- (0,1) node{} -- (\x,0) node{};}
\foreach \x in {.25,1,2}
	{\draw (1.5,2) node{} -- (1,1) node{} -- (\x,0) node{};}
\foreach \x in {1,2,2.75}
	{\draw (1.5,2) node{} -- (2,1) node{} -- (\x,0) node{};}
\foreach \x in {1,2,3.5}
	{\draw (1.5,2) node{} -- (3,1) node{} -- (\x,0) node{};}
\node at (1.5,2) [label=right: ${x}$, fill=black] {};

\end{tikzpicture}
%}

\caption{The $B_2^{\rm inc}(x)$-structure of $G_2$}
\label{fig:G_2}
    \end{subfigure}
   \end{center}
   \caption{Two graphs $G_1$ and $G_2$ with different $B_2^{\rm inc}(x)$-structures share the same curvature matrix. For example, $G_1$ can be the $4$-dimensional cube $\Q^4$.}
   \label{fig:example_same_A_inf}
\end{figure}

On the other hand, the curvature matrix $A_\infty(x)$ contains more information than the curvature function $\K_{G,x}$, and $A_\infty(x)$ cannot be recovered from $\K_{G,x}$. For example, it is shown below that the non-weighted cube $\Q^3$ and complete bipartite graph $K_{3,3}$ share the same curvature function, while having different curvature matrices.

\begin{multicols}{2}
For any vertex $x$ in $G=\Q^3$:
\begin{align*}
A_N^{G}(x) &= \begin{pmatrix} 2&&\\ &2&\\ &&2 \end{pmatrix} - \frac{2}{N}J_3,\\
\sigma(A_N^{G} (x)) &= \{ 2-\frac{6}{N}, 2, 2\}, \\
\K_{G,x}(N) &= 2-\frac{6}{N}.
\end{align*}

\columnbreak
For any vertex $x$ in $H=K_{3,3}$:
\begin{align*}
A_N^{H}(x) &= \begin{pmatrix} 8/3&-1/3&-1/3\\ -1/3&8/3&-1/3\\ -1/3&-1/3&8/3 \end{pmatrix} - \frac{2}{N}J_3,\\
\sigma(A_N^{H} (x)) &= \{ 2-\frac{6}{N}, 3, 3\}, \\
\K_{H,x}(N) &= 2-\frac{6}{N}.
\end{align*}
\end{multicols}

\end{remark}

%\begin{remark} \label{rem:cube_K33}
%The curvature matrix $A_\infty(x)$ contains more information than the curvature function $\K_{G,x}$, and $A_\infty(x)$ cannot be recovered from $\K_{G,x}$. As an example, it is shown below that the (non-weighted) cube $\Q^3$ and complete bipartite $K_{3,3}$ share the same curvature functions, while having different curvature matrices.
%
%For any vertex $x$ in $G=\Q^3$:
%\begin{align*}
%A_N^{G}(x) &= A_\infty^{G}(x) - \frac{2}{N}\v_0\v_0^\top = \begin{pmatrix} 2&&\\ &2&\\ &&2 \end{pmatrix} - \frac{2}{N}J_3,\\
%\sigma(A_N^{G} (x)) &= \{ 2-\frac{6}{N}, 2, 2\}, \\
%\K_{G,x}(N) &= 2-\frac{6}{N}.
%\end{align*}
%
%For any vertex $x$ in $H=K_{3,3}$:
%\begin{align*}
%A_N^{H}(x) &= A_\infty^{H}(x) - \frac{2}{N}\v_0\v_0^\top = \begin{pmatrix} 8/3&-1/3&-1/3\\ -1/3&8/3&-1/3\\ -1/3&-1/3&8/3 \end{pmatrix} - \frac{2}{N}J_3,\\
%\sigma(A_N^{H} (x)) &= \{ 2-\frac{6}{N}, 3, 3\}, \\
%\K_{H,x}(N) &= 2-\frac{6}{N}.
%\end{align*}
%
%\end{remark}
%%%%%%%%%%%%%%%%%%%%%%%%%%%%%%%%%%%%%%%%%%%%%%%%%%%%%
\section{Properties of the curvature function $\K_{G,x}$} \label{sect:curv_fct_properties}
%%%%%%%%%%%%%%%%%%%%%%%%%%%%%%%%%%%%%%%%%%%%%%%%%%%%%

This section is devoted to the proof of Theorem \ref{thm:continuous_and_threshold} about properties of the curvature function $\K_{G,x}:(0,\infty] \to \IR$, which will be divided into small steps.

\begin{proposition} \label{prop:curv_fct_cont_limits}
The curvature function $\K_{G,x}:(0,\infty] \to \IR$ is continuous, monotone increasing and concave with $\lim_{N\to 0} \K_{G,x}(N)=-\infty$ and $\lim_{N\to \infty} \K_{G,x}(N) <\infty$.
\end{proposition}

\begin{proof}
It is known that the zeros of a polynomial are continuous functions of the coefficients of the polynomial (see, e.g., \cite[Theorem (1,4)]{Marden}). In particular for the characteristic polynomial in $\lambda$, namely $\det(A_N-\lambda {\rm Id})$, it means the ordered set of eigenvalues of $A_N$, respecting their multiplicities, are continuous in $N$. In particular, $\K_{G,x}(N)=\lambda_{\min{}}(A_N)$ is continuous in $N$.

Monotonicity and concavity of $\K_{G,x}$ employ the crucial fact that, for symmetric matrices $A$ and $B$,
\[
\lambda_{\min{}}(A+B) = \inf_{v\not=0} \frac{v^\top (A+B) v}{v^\top v} \ge \inf_{v\not=0} \frac{v^\top A v}{v^\top v} + \inf_{v\not=0} \frac{v^\top B v}{v^\top v} = \lambda_{\min{}}(A)+\lambda_{\min{}}(B),
\]
and the inequality holds with equality iff $A$ and $B$ share an eigenvector corresponding to their minimal eigenvalues. Recall also that $\mathbf{v}_0\mathbf{v}_0^\top$ is a rank one matrix with the only nontrivial eigenvalue $\mathbf{v}^\top_0\mathbf{v}_0>0$, so $\lambda_{\max{}}(\mathbf{v}_0\mathbf{v}_0^\top)=\mathbf{v}^\top_0\mathbf{v}_0$ and $\lambda_{\min{}}(\mathbf{v}_0\mathbf{v}_0^\top)=0$.

For $0<N'<N\le \infty$, we have
\begin{align} \label{eq:monotone_ineq}
\lambda_{\min{}}(A_N) &=\lambda_{\min{}}\left(A_{N'}+\Big(\frac{2}{N'}-\frac{2}{N}\Big)\mathbf{v}_0\mathbf{v}_0^\top\right) \nonumber\\
&\ge \lambda_{\min{}}(A_{N'})+ \lambda_{\min{}} \Big( \underbrace{(\frac{2}{N'}-\frac{2}{N})}_{>0} \mathbf{v}_0\mathbf{v}_0^\top\Big)= \lambda_{\min{}}(A_{N'}),
\end{align}
Similarly, for $0<N'<N\le \infty$ and $\alpha\in (0,1)$, we have
\begin{align*}
\lambda_{\min{}}(A_{\alpha N+(1-\alpha)N'})
&= \lambda_{\min{}} \Big(\alpha A_{N} +(1-\alpha)A_{N'} +
2\underbrace{(\frac{\alpha}{N}+\frac{1-\alpha}{N'}-\frac{1}{\alpha N+(1-\alpha)N'})}_{> 0} \mathbf{v}_0\mathbf{v}_0^\top \Big) \\
&\ge \alpha \lambda_{\min{}}( A_{N}) + (1-\alpha)\lambda_{\min{}}( A_{N'}),
\end{align*}

To derive $\K_{G,x}(\infty)= \lim_{N\to \infty} \K_{G,x}(N) < \infty$ and $\lim_{N\to 0} \K_{G,x}(N) = -\infty $, we argue that
\begin{align*}
\lambda_{\min{}}(A_N) =\lambda_{\min{}}\left(A_{\infty}-\frac{2}{N}\mathbf{v}_0\mathbf{v}_0^\top\right)
 \to \lambda_{\min{}}(A_\infty) \qquad \text{as } N\to \infty,
\end{align*}
and
\[\lambda_{\min{}}(A_N) \le \|A_{\infty}\|+\lambda_{\min{}}\bigl(-\frac{2}{N}\mathbf{v}_0\mathbf{v}_0^\top\bigr) = \|A_{\infty} \| -\frac{2}{N}\mathbf{v}_0^\top\mathbf{v}_0  \to -\infty \qquad \text{as } N\to 0, \] where $\|\cdot\|$ denotes the operator norm.
%\begin{align*}
%\lambda_{\min{}}(A_N) &=\lambda_{\min{}}\left(A_{\infty}-\frac{2}{N}\mathbf{v}_0\mathbf{v}_0^\top\right)\\
%&\le \|A_{\infty}\|+\lambda_{\min{}}\left(-\frac{2}{N}\mathbf{v}_0\mathbf{v}_0^\top\right) \\
%&= \|A_{\infty} \| -\frac{2}{N}\mathbf{v}_0^\top\mathbf{v}_0  \to -\infty \qquad \text{as } N\to 0,
%\end{align*}
\end{proof}

\begin{lemma} \label{lem:notsimple_eigenvalue}
If $\lambda_{\min{}}(A_{N'})$ is not simple for some $N'\in (0,\infty]$, then $\lambda_{\min{}}(A_{N})=\lambda_{\min{}}(A_{N'})$ for all $N\in [N',\infty]$. In other words, $\K_{G,x}$ is constant on $[N',\infty]$.
\end{lemma}

\begin{proof}
Assume that $\lambda_{\min{}}(A_{N'})$ is not simple, that is, the minimal eigenspace $E_{\min{}}(A_{N'})$ has dimension at least $2$. We first argue that there exists a nonzero $w \in E_{\min{}}(A_{N'})$ such that $w\perp \mathbf{v}_0$. Consider any two linearly independent vectors $v_1 = a_1\mathbf{v}_0 + b_1w_1$ and $v_2 = a_2\mathbf{v}_0 + b_2w_2$ in $E_{\min{}}(A_{N'})$ with $w_1\perp \mathbf{v}_0$ and $w_2\perp \mathbf{v}_0$.  In case $a_1=0$ or $a_2=0$, we immediately obtain such a vector $w$. In case $a_1\not=0$ and $a_2\not=0$,  the vector $\frac{1}{a_1}v_1-\frac{1}{a_2}v_2$ represents such a vector $w$.

Since $w \perp \mathbf{v}_0$, it lies in the minimal eigenspace $E_{\min{}}(\mathbf{v}_0\mathbf{v}_0^\top)$ whose minimal eigenvalue is zero. This means $w\in E_{\min{}}(A_{N'}) \cap E_{\min{}}(\mathbf{v}_0\mathbf{v}_0^\top)$, so the inequality \eqref{eq:monotone_ineq} holds with equality, i.e.,
$\lambda_{\min{}}(A_{N})=\lambda_{\min{}}(A_{N'})$ for all $N\in [N',\infty]$.
\end{proof}

\begin{lemma} \label{lem:analytic_at_simple}
If $\lambda_{\min{}}(A_{N})$ is simple for some $N\in (0,\infty]$ , then $\K_{G,x}$ is analytic in a small neighbourhood of $N$.
\end{lemma}

\begin{proof}
The idea is to prove analyticity by using the implicit function theorem. More precisely, we aim to apply \cite[Theorem 6.1.2]{Krantz-Parks}.
Consider the matrix-valued function $A(t)=A_{1/t}$, and denote $\lambda_0(t) \le \lambda_1(t) \le ... \le \lambda_{m-1}(t)$ to be all eigenvalues of $A(t)$. Let $t_0=\frac{1}{N}$ and assume that $\lambda_0(t_0)=\lambda_{\min{}}(A_{N})$ is simple. Consider the following polynomial in $t$ and $\lambda$:
\[ F(t,\lambda):=\det(A(t_0+t) -(\lambda_{0}(t_0)+\lambda){\rm Id})=\sum_{i,j} a_{i,j} t^i\lambda^j. \]
The characteristic polynomial factorization gives
\[ F(0,\lambda)=\det(A(t_0)-(\lambda_{0}(t_0)+\lambda){\rm Id})=\prod_{i=0}^{m-1} (\lambda_i(t_0) - (\lambda_0(t_0)+\lambda))=\lambda \prod_{i=1}^{m-1} (\lambda_i(t_0)-\lambda_0(t_0)-\lambda),\]
which means $a_{0,0}=0$, and $a_{0,1}\not = 0$ since $\lambda_0(t_0)\not=\lambda_i(t_0)$ for $i\ge 1$.
The analytic implicit function theorem asserts that there exists an analytic function $\lambda(t)$ around $t=0$ such that $\lambda(0)=0$ and $F(t, \lambda(t))=0$ for all $t$ near $0$, that is, $\lambda_0(t_0)+\lambda(t)$ is an eigenvalue of $A(t_0+t)$. Moreover, the assumption that $\lambda_0(t_0)$ is a simple and smallest eigenvalue of $A(t_0)$ implies that $\lambda_0(t_0)+\lambda(t)$ stays the smallest eigenvalue of $A(t_0+t)$ for $t$ near $0$.
\end{proof}

\begin{lemma} \label{lem:threshold_nonsimp}
If $\lambda_{\min{}}(A_{N_1})$ is not simple for some $N_1\in (0,\infty]$, then there exists the smallest such $N_1$, and consequently $K_{G,x}$ is analytic, strictly monotone increasing and strictly concave on $(0,N_1]$, and constant on $[N_1,\infty]$.
\end{lemma}

\begin{proof}
Consider the set
\[ \N_{\rm ns} := \{N\in (0,\infty]:\ \lambda_{\min{}}(A_N) \text{ is not simple}\}, \]
and denote $N_1:=\inf \N_{\rm ns}$.

We know from Lemma \ref{lem:notsimple_eigenvalue} that $\K_{G,x}$ is constant on $[N,\infty]$ for all $N\in \N_{\rm ns}$. Therefore, $\K_{G,x}$ is constant on $(N_1,\infty]$. Note that $N_1>0$; otherwise $\K_{G,x}$ is constant on the whole interval $(0,\infty]$, which contradicts to the fact from Proposition \ref{prop:curv_fct_cont_limits} that $\lim_{N\to 0} \K_{G,x}(N)=-\infty$.

If $\lambda_{\min{}}(A_{N_1})$ were simple, then $\lambda_{\min{}}(A_{N})$ would also be simple for all $N$ in a small neighbourhood of $N_1$. This contradicts to the definition of $N_1$. Therefore, $\lambda_{\min{}}(A_{N_1})$ is not simple, and $N_1=\min \N_{\rm ns}$.

Since $\lambda_{\min{}}(A_{N})$ is simple for all $N\in(0,N_1)$, we know from Lemma \ref{lem:analytic_at_simple} that $\K_{G,x}$ is analytic on $(0,N_1)$.
Recall also from Proposition \ref{prop:curv_fct_cont_limits} that $\K_{G,x}$ is concave and monotone increasing. If $\K_{G,x}$ were not strictly concave on $(0,N_1)$, this would mean $\K_{G,x}$ is linear on some interval $[a,b] \subset (0,N_1)$. Then the analyticity of $\K_{G,x}$ on $(0,N_1)$ would then imply that $\K_{G,x}$ is linear on the entire interval $(0,N_1)$, which contradicts to the fact that $\lim_{N\to 0} \K_{G,x}(N)=-\infty$. Thus $\K_{G,x}$ is indeed strictly concave on $(0,N_1)$, and consequently it is strictly monotone increasing on $(0,N_1)$. This finishes the proof of Lemma \ref{lem:threshold_nonsimp}.
\end{proof}

By combining Proposition \ref{prop:curv_fct_cont_limits} and Lemmas \ref{lem:analytic_at_simple} and \ref{lem:threshold_nonsimp}, we can conclude Theorem \ref{thm:continuous_and_threshold} with the description of the threshold $N_1\in(0,\infty]$, namely $N_1=\min \{N\in (0,\infty]:\ \lambda_{\min{}}(A_N) \text{ is not simple}\}$ (and $N=\infty$ in case this set is empty).

Let us end this section with the proof of Proposition \ref{cor: N_threshold} about the uniqueness of the threshold $N_0$ such that $\K_{G,x}(N_0)=0$, which is asserted by the intermediate value theorem for the continuous curvature function $\K_{G,x}:(0,\infty] \to \IR$.

\begin{proof}[Proof of Proposition \ref{cor: N_threshold}]
Since $\K_{G,x}(\infty)>0$ (by assumption) and $\lim_{N\to 0}\K_{G,x}(N) = -\infty$, the intermediate value theorem asserts that there exists an $N_0\in (0,\infty)$ such that $\K_{G,x}(N_0)=0$. This implies $\det A_{N_0}=0$.

Furthermore, $\K_{G,x}(\infty)>0$ means $\det A_\infty>0$ and $A_\infty$ is invertible. The matrix determinant formula then gives
\begin{align} \label{eq:matrix_det_formula}
0=\det A_{N_0} =\det(A_\infty-\frac{2}{N_0}\mathbf{v}_0\mathbf{v}_0^\top)=(1-\frac{2}{N_0}\mathbf{v}_0^\top A_\infty^{-1}\mathbf{v}_0) \det A_\infty.
\end{align}
Therefore, $N_0$ is uniquely given by
$N_0=2\mathbf{v}_0^\top A_\infty^{-1}\mathbf{v}_0$.
\end{proof}

%%%%%%%%%%%%%%%%%%%%%%%%%%%%%%%%%%%%%%%%%%%%%%%%%%%%%
\section{Curvature bounds and curvature sharpness} \label{sect:bounds}
%%%%%%%%%%%%%%%%%%%%%%%%%%%%%%%%%%%%%%%%%%%%%%%%%%%%%
\begin{proof}[Proof of Theorem \ref{thm:lower_upper_bound}]
We derive the lower curvature bound via the Rayleigh quotient as follows:
\begin{align*}
\K_{G,x}(N) = \inf_{v\not=0} \frac{v^\top (A_\infty - \frac{2}{N}\mathbf{v}_0\mathbf{v}_0^\top) v}{v^\top v}
\ge \inf_{v\not=0} \frac{v^\top A_\infty v}{v^\top v} - \frac{2}{N} \sup_{v\not=0}  \frac{v^\top \mathbf{v}_0\mathbf{v}_0^\top v}{v^\top v}
= \K_{G,x}(\infty) -\frac{2}{N}\mathbf{v}_0^\top\mathbf{v}_0,
\end{align*}
where $\mathbf{v}_0^\top\mathbf{v}_0=\sum_{y\in S_1(x)} p_{xy} = \frac{d_x}{\mu_x}$.

On the other hand, the upper curvature bound can be derived as
\begin{equation} \label{eq:upper_bound_Rayleigh}
\K_{G,x}(N) \le  \frac{\mathbf{v}_0^\top A_N \mathbf{v}_0}{\mathbf{v}_0^\top \mathbf{v}_0} = \frac{\mathbf{v}_0^\top (A_\infty - \frac{2}{N}\mathbf{v}_0\mathbf{v}_0^\top) \mathbf{v}_0}{\mathbf{v}_0^\top \mathbf{v}_0}
= \frac{\mathbf{v}_0^\top A_\infty \mathbf{v}_0}{\mathbf{v}_0^\top \mathbf{v}_0} - \frac{2}{N}\mathbf{v}_0^\top\mathbf{v}_0.
\end{equation}
Lemma \ref{lem:K0infty} in Appendix \ref{sect:appendix} confirms that
\[ \K^{0}_{\infty}(x): = \frac{\mathbf{v}_0^\top A_\infty \mathbf{v}_0}{\mathbf{v}_0^\top \mathbf{v}_0} = \frac{1}{2}\left( \frac{d_x}{\mu_x}+ 3\frac{\mu_x}{d_x} p_{xx}^{(2)} - \frac{\mu_x}{d_x} \sum_{z\in S_2(x)} p_{xz}^{(2)}\right) . \]
\end{proof}

\begin{remark}\label{remark:K0infty}
In the case of non-weighted graphs, the quantity $\K_{\infty}^0(x)$ reduces to the one in \cite[Definition 3.2]{CLP20}. Indeed, we have in that case
\begin{align*}
\K_\infty^0(x)=\frac{1}{2}\left(d_x+3-\frac{1}{d_x}\sum_{y\in S_1(x)}d_y^+\right)=2+\frac{1}{2}\left(d_x-\frac{1}{d_x}\sum_{y\in S_1(x)}d_y\right)+\frac{\sharp_{\triangle}(x)}{d_x},
\end{align*}
where $d_y^+$ is the out-degree of $y\in S_1(x)$ (i.e., the number of neighbours of $y$ in $S_2(x)$) and $\sharp_{\triangle}(x)$ denotes the number of triangles containing $x$.
\end{remark}

%%%%%%%%%%%%%%%%%%%%%%%%%%%%%%%%%%%%%%%%%%%%%%%%%%%%%
\section{Relations between the spectrum of the curvature matrix $A_\infty$ and the curvature function $\K_{G,x}$} \label{sect:spectral_relation}
%%%%%%%%%%%%%%%%%%%%%%%%%%%%%%%%%%%%%%%%%%%%%%%%%%%%%

%In order to justify the properties of curvature sharpness in Proposition \ref{prop:curv_sharp_shape}, we need to argue via the relation that $\mathbf{v}_0$ is an eigenvector of the curvature matrix $A_\infty$. This means we will give the proof of Proposition \ref{prop:v0_eigenvalue} together with this proposition.

\begin{proof}[Proof of Proposition \ref{prop:v0_eigenvalue} and Proposition \ref{prop:curv_sharp_shape}] \ \\
The vertex $x$ is $N$-curvature sharp if and only if the upper bound \eqref{eq:upper_bound_Rayleigh}: $\lambda_{\min{}}(A_N)
\le \frac{\mathbf{v}_0^\top A_N \mathbf{v}_0}{\mathbf{v}_0^\top \mathbf{v}_0}$ holds with equality, which happens if and only if $\mathbf{v}_0$ is in the minimal eigenspace $E_{\min{}}(A_N)$. In particular, $x$ is $\infty$-curvature sharp if and only if $\mathbf{v}_0 \in E_{\min{}}(A_\infty)$. This proves Proposition \ref{prop:v0_eigenvalue} (ii).

Assume $x$ is $N_1$-curvature sharp for some $N_1\in (0,\infty]$. Then $A_{N_1} \mathbf{v}_0= \lambda_{\min{}}(A_{N_1}) \mathbf{v}_0$, which implies $A_\infty \mathbf{v}_0 = (\lambda_{\min{}}(A_{N_1})+\frac{2}{N_1}\mathbf{v}_0^\top\mathbf{v}_0) \mathbf{v}_0$, that is, $\mathbf{v}_0$ is an eigenvector of $A_\infty$.

Conversely, assume $\mathbf{v}_0$ is an eigenvector of $A_\infty$, that is, $A_\infty \mathbf{v}_0 = \lambda \mathbf{v}_0$ for some $\lambda\in \IR$. Denote the spectrum of $A_\infty$ by $\sigma(A_\infty)=\{\lambda, \lambda_1,...,\lambda_{m-1}\}$ with $\lambda_1 \le ... \le \lambda_{m-1}$. Consider $A_\infty v_i=\lambda_i v_i$ where all eigenvectors $v_i$ of $A_\infty$ (different from $\mathbf{v}_0$) are chosen to be orthogonal to $\mathbf{v}_0$. We then obtain for any $N$,
\begin{align*}
A_N \mathbf{v}_0 &= (A_\infty-\frac{2}{N}\mathbf{v}_0\mathbf{v}_0^\top) \mathbf{v}_0 = (\lambda-\frac{2}{N}\mathbf{v}_0^\top\mathbf{v}_0) \mathbf{v}_0;\\
A_N v_i &= (A_\infty-\frac{2}{N}\mathbf{v}_0\mathbf{v}_0^\top) v_i=A_\infty v_i =\lambda_i v_i \qquad \forall 1\le i <m,
\end{align*}
which mean its spectrum is $\sigma(A_N)=\{\lambda-\frac{2}{N}\mathbf{v}_0^\top\mathbf{v}_0, \lambda_1,...,\lambda_{m-1}\}$.

We choose the threshold $N_1=\frac{2\mathbf{v}_0^\top\mathbf{v}_0}{\lambda - \lambda_1}$ in case $\lambda \ge \lambda_1$ (and choose $N_1=\infty$ if $\lambda < \lambda_1$), so that
\begin{equation*}
\lambda_{\min{}}(A_N)  =
\begin{cases}
\lambda-\frac{2}{N}\mathbf{v}_0^\top\mathbf{v}_0 &\text{ if } N \le N_1, \\
\lambda_1 &\text{ if } N \ge N_1.
\end{cases}
\end{equation*}
This means for all $N\le N_1$, $\mathbf{v}_0 \in E_{\min{}}(A_N)$, that is, $x$ is curvature sharp on $(0,N_1]$. This proves Proposition \ref{prop:v0_eigenvalue} (i). Furthermore, for all $N \ge N_1$, $\lambda_{\min{}}(A_N)=\lambda_1=\lambda_{\min{}}(A_\infty)$, that is, $\K_{G,x}$ is constant on $[N_1,\infty]$. This proves the two forward statements of Proposition \ref{prop:curv_sharp_shape}.

To verify the converse statement in Proposition \ref{prop:curv_sharp_shape}, suppose that $\K_{G,x}(N)=c-\frac{2}{N}\frac{d_x}{\mu_x}$ for all $N\in (N',N'')$ and hence at $N=N',N''$ by continuity of $\K_{G,x}$. We observe that
\begin{align*}
c-\frac{2}{N''}\frac{d_x}{\mu_x}
= \lambda_{\min{}}(A_{N''}) &=\lambda_{\min{}}\left(A_{N'}+\Big(\frac{2}{N'}-\frac{2}{N''}\Big)\v_0\v_0^\top\right) \\
&=\inf_{v\not=0} \left( \frac{v^\top A_{N'} v}{v^\top v} + \Big(\frac{2}{N'}-\frac{2}{N''}\Big) \frac{v^\top \v_0\v_0^\top v}{v^\top v} \right) \\
&\le \inf_{v\not=0} \frac{v^\top A_{N'} v}{v^\top v} + \Big(\frac{2}{N'}-\frac{2}{N''}\Big) \v_0^\top \v_0 \\
&=  \lambda_{\min{}}(A_{N'}) + \Big(\frac{2}{N'}-\frac{2}{N''}\Big) \frac{d_x}{\mu_x}
= c-\frac{2}{N''}\frac{d_x}{\mu_x},
\end{align*}
so the inequality holds with equality, which occurs when $\v_0\in E_{\min{}}(A_{N'})$. Consequently, it holds that $\v_0\in E_{\min{}}(A_{N''})$. So $x$ is $N''$-curvature sharp as desired.
\end{proof}

The next result is an interesting observation about the non-smallest eigenvalues of $A_N$, which is not included in the Introduction.

\begin{corollary} \label{cor:non-smallest_eigen_positive}
If $\K_{G,x}(\infty)>0$, then all of the non-smallest eigenvalues of $A_N$ are strictly positive for all dimensions $N\in (0,\infty]$.
\end{corollary}

\begin{proof}
Let $\lambda_i(A_{N})$ denote the $i$-th smallest eigenvalue of $A_{N}$ (respecting multiplicity).
Assume for the sake of contradiction that there exist $N'\in (0,\infty)$ and $i\ge 2$ such that $\lambda_i(A_{N'}) \not=\lambda_{\min{}}(A_{N'})$ and $\lambda_i(A_{N'})<0$ . We also know from $\K_{G,x}(\infty)>0$ that $\lambda_i(A_{\infty}) >0$. Since $\lambda_i(A_{N})$ is continuous on $N$, the intermediate value theorem implies that $\lambda_i(A_{\hat{N}_0})=0$ for some $\hat{N}_0\in (N',\infty)$, and hence $\det(A_{\hat{N}_0})=0$. The matrix determinant formula $0=\det A_{\hat{N}_0} =(1-\frac{2}{\hat{N}_0}\mathbf{v}_0^\top A_\infty^{-1}\mathbf{v}_0) \det A_\infty$  with $\det A_\infty>0$ (because $\K_{G,x}(\infty)>0$) asserts that $\hat{N}_0=2\mathbf{v}_0^\top A_\infty^{-1}\mathbf{v}_0$, which is the same threshold as $N_0$ in Proposition \ref{cor: N_threshold}. In other words, $\lambda_{\min{}}(A_{\hat{N}_0})=0=\lambda_{i}(A_{\hat{N}_0})$ is not simple. By Lemma \ref{lem:threshold_nonsimp}, $\K_{G,x}$ must then be constant on $[\hat{N}_0,\infty)$, which is contradiction to the fact that $\K_{G,x}(\hat{N}_0) = 0< \K_{G,x}(\infty)$.
\end{proof}

%Next, we discuss the situation where $\mathbf{v}_0 \perp E_{\min{}}(A_\infty)$ but $\mathbf{v}_0 \not \in \sigma(A_\infty)$.
Proposition \ref{prop:curv_sharp_shape} raises the question whether there exists a graph with a vertex $x$ which is not curvature sharp for any finite $N$ but nevertheless its curvature function is constant near infinity. The following example provides the answer.
\begin{example} \label{ex:Emin_perp_v0}
%the curvature function $\K_{G,x}$ has a finite threshold $N_1$, where $\K_{G,x}$ is constant on $[N_1,\infty]$, and strictly monotone increasing but not curvature-sharp on $(0,N_1]$ . The idea is to consider the Cartesian product $G=G_1\times G_2$ of the two non-weighted graphs $G_1=(V_1,E_1)$ and $G_2=(V_2,E_2)$. The analysis on the Cartesian product is discussed in general in the next section.
%
%For $i\in \{1,2\}$, let $x_i\in V_i$ be a $S_1$-out regular vertex of $G_i$. By Theorem \ref{thm:curv_sharp_criterion} (or \cite[Corollary 5.10]{CLP20}) and Proposition \ref{prop:curv_sharp_shape}, there exists a threshold $N_i\in (0,\infty]$ such that $\K_{G_i,x_i}$ is constant on $[N_i,\infty]$.
%Then the Cartesian product result (see \cite[Theorem 7.9]{CLP20}) at the vertex $x:=(x_1,x_2)$, the curvature $\K_{G_1\times G_2,(x_1,x_2)}$ is constant on $[N_1+N_2,\infty]$.
%
%Furthermore, we can make a particular choice for graphs $G_i$ and vertices $x_i$ so that $(x_1,x_2)$ is not $S_1$-out regular and the threshold $N_i<\infty$. This means $(x_1,x_2)$ is not $N$-curvature sharp for any dimension $N$ (since curvature sharpness is equivalent to $S_1$-out regularity in a non-weighted graph \cite[Corollary 5.10]{CLP20}), but $\K_{G_1\times G_2,(x_1,x_2)}$ is constant on a nontrivial interval $[N_1+N_2,\infty]$ as desired.

We consider the Cartesian product $P_3 \times P_2$, where $P_n$ is the path containing $n$ vertices.

\begin{figure}[h!]
\centering
\tikzstyle{every node}=[circle, draw, fill=black!20, inner sep=0pt, minimum width=4pt]
\begin{tikzpicture}[scale=1.5]

\foreach \y in {0,1}
	{\draw (0,\y) node{} -- (1,\y) node{} -- (2,\y) node{};}
\foreach \x in {0,1,2}
	{\draw (\x,0) node{} -- (\x,1) node{};}
%\node at (0,0) [label=below: ${x}$, fill=black] {};
\node at (1,0) [label=below: ${x}$, fill=black] {};
\end{tikzpicture}
\caption{Cartesian product of $P_3$ and $P_2$}
\label{fig:product_example}
\end{figure}
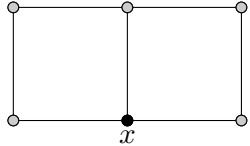

The curvature matrix at $x$ is given by \begin{align*}
A_\infty(x) =
	\left(\begin{array}{ccc}
	2 & 0 & 0 \\
	0 & 1.5 & 1  \\
	0 & 1 & 1.5
	\end{array}
	\right),
\end{align*} and it has the smallest eigenvalue of $0.5$. The vector $\v_0=(1 \ 1 \ 1)^\top$ is not an eigenvector, but it is perpendicular to the minimal eigenspace $E_{\min{}}(A_\infty(x))={\rm span}\ (0 \ 1 \ -1)^\top$.

The rank one perturbation $A_N(x)=A_\infty(x)- \frac{2}{N} \v_0\v_0^\top = A_\infty(x)- \frac{2}{N}J_{3}$ has its spectrum equal to $\sigma(A_N(x)) = \{ \frac{1}{2}, \ \frac{9}{4}-\frac{3}{N} \pm \sqrt{\frac{1}{16}-\frac{1}{2N}+\frac{9}{N^2}} \}$.
Therefore, the curvature function at $x$ is given by
\begin{equation*}
\K_{P_3\times P_2}(x) =
 \begin{cases}
\frac{9}{4}-\frac{3}{N} - \sqrt{\frac{1}{16}-\frac{1}{2N}+\frac{9}{N^2}} &\text{ if } N \in (0,\frac{10}{3}] \\
\frac{1}{2} &\text{ if } N \in [\frac{10}{3},\infty].
\end{cases}
\end{equation*}

%\begin{align*}
%A_N(x) =
%	\left(\begin{array}{cc}
%	2 & 0 \\
%	0 & 1.5
%	\end{array}
%	\right) - \frac{2}{N}
%	\left(\begin{array}{ccc}
%	1 & 1  \\
%	1 & 1
%	\end{array}
%	\right),
%\end{align*}
%and its spectrum is equal to $\sigma(A_N(x)) =...$.
\end{example}

%%%%%%%%%%%%%%%%%%%%%%%%%%%%%%%%%%%%%%%%%%%%%%%%%%%%%
\section{Curvature of Cartesian product of graphs} \label{sect:Cartesian}
%%%%%%%%%%%%%%%%%%%%%%%%%%%%%%%%%%%%%%%%%%%%%%%%%%%%%
Given two weighted graphs $G,G'$ and two fixed positive numbers $\alpha,\beta\in \IR^{+}$, the weighted Cartesian product $G\times_{\alpha,\beta} G'$ is defined with the following weight function and vertex measure: for $x,y\in G$ and $x',y'\in G'$,
\begin{align*}
w_{(x,x')(y,x')} &:= \alpha w_{xy} \mu_{x'},\\
w_{(x,x')(x,y')} &:= \beta w_{x'y'} \mu_{x},\\
\mu_{(x,x')} &:=\mu_x \mu_{x'}.
\end{align*}
One can translate the above definition into the transition rate $p$ as
\begin{align*}
p_{(x,x')(y,x')} &= \alpha \frac{w_{xy}\mu_{x'}}{\mu_x \mu_{x'}} = \alpha p_{xy},\\
p_{(x,x')(x,y')} &= \beta p_{x'y'},\\
\frac{d_{(x,x')}}{\mu_{(x,x')}} &= \sum_{y} p_{(x,x')(y,x')} + \sum_{y'} p_{(x,x')(x,y')} = \alpha \frac{d_{x}}{\mu_{x}}+ \beta \frac{d_{x'}}{\mu_{x'}}.
\end{align*}
Here we use the same symbols $w,\mu,p$ and $d$ for all graphs $G$, $G'$ and its product, where the associated graph can be determined from the input vertices. With this idea, we also use the notations $A_\infty(\cdot), A_N(\cdot)$ and $Q(\cdot)$. This simplifies our notations without making them ambiguous.

\begin{proof}[Proof of Theorem \ref{thm:cartesian_curv_matrix}]

Now the central vertex is $(x,x')$ with horizontal neighbours $(y,x')$ for $y\in S_1(x)$ and vertical neighbours $(x,y')$ for $y'\in S_1(x')$.
Note also that $(y,x')$ and $(x,y')$ are not adjacent but sharing one common neighbour in $S_2$, namely $(y,y')$.  On the other hand, the vertex $(y,y')$ has exactly two neighbours in $S_1$, namely $(y,x')$ and $(x,y')$. The transition rate on each edge are presented in the following scheme.

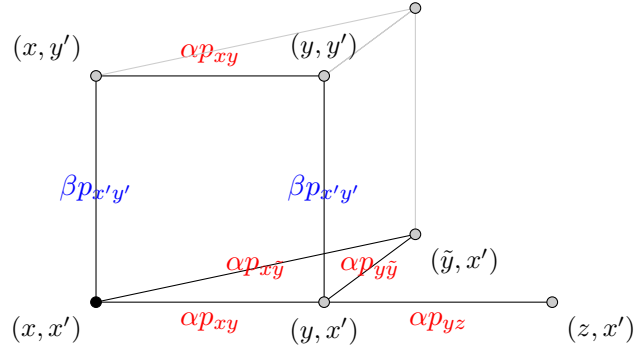
\begin{figure}[!htp]
\centering
\tikzset{vertex/.style={circle, draw, fill=black!20, inner sep=0pt, minimum width=4pt}}
%\tikzstyle{every node}=[circle, draw, fill=black!20, inner sep=0pt, minimum width=4pt]
\begin{tikzpicture}[scale=3.0]

\draw (0,0) -- (1,0) node[midway, below, red]{$\alpha p_{xy}$}
		 -- (1,1) node[midway, blue]{$\beta p_{x'y'}$}
		 -- (0,1) node[midway, above, red]{$\alpha p_{xy}$}
		 -- (0,0) node[midway, blue]{$\beta p_{x'y'}$};
\draw (0,0) -- (1.4,0.3) node[midway, red]{$\alpha p_{x\tilde{y}}$}
		 -- (1,0) node[midway, red]{$\alpha p_{y\tilde{y}}$}
		 -- (2,0) node[midway, below, red]{$\alpha p_{yz}$} ;
\draw[black!20] (0,1) -- (1.4,1.3) -- (1,1) -- (1.4,1.3) -- (1.4,0.3);

\node at (0,0) [vertex, label={[label distance=0mm]225: \small $(x,x')$}, fill=black] {};
\node at (1,0) [vertex, label={[label distance=0mm]270: \small $(y,x')$}] {};
\node at (0,1) [vertex, label={[label distance=0mm]135: \small $(x,y')$}] {};
\node at (1,1) [vertex, label={[label distance=0mm]90: \small $(y,y')$}] {};
\node at (2,0) [vertex, label={[label distance=0mm]315: \small $(z,x')$}] {};
\node at (1.4,0.3) [vertex, label={[label distance=0mm]315: \small $(\tilde{y},x')$}] {};
\node at (1.4,1.3) [vertex] {};

\end{tikzpicture}
\caption{The scheme showing a horizontal neighbour and a vertical neighbour of the central vertex $(x,x')$ in the Cartesian product $G \times_{\alpha,\beta} G'$ and transition rate $p$ on each edge.}
\label{fig:product_scheme}
\end{figure}

For $y\in S_1(x)$, we have from \eqref{eq:Q_ondiag} that
\begin{align*}
&4Q((x,x'))_{(y,x')(y,x')} \\
& = 2p_{(x,x')(y,x')}^2+3p_{(x,x')(y,x')}p_{(y,x')(x,x')} - \frac{d_{(x,x')}}{\mu_{(x,x')}}p_{(x,x')(y,x')}\\
&\phantom{=}+
	3p_{(x,x')(y,x')} \Bigl( \sum_{\mathclap{z\in S_2(x)}} p_{(y,x')(z,x')} +  \sum_{\mathclap{y'\in S_1(x')}} p_{(y,x')(y,y')} \Bigr) \\
& \phantom{=} + \sum_{\tilde{y}\in S_1(x)} (3p_{(x,x')(y,x')}p_{(y,x')(\tilde{y},x')} + p_{(x,x')(\tilde{y},x')}p_{(\tilde{y},x')(y,x')}) \\
& \phantom{=} - 4  \sum_{z\in S_2(x)} \frac{p_{(x,x')(y,x')}^2 p_{(y,x')(z,x')}^2}{\sum_{y\in S_1(x)}p_{(x,x')(y,x')} p_{(y,x')(z,x')}} \\
& \phantom{=}- 4\sum_{y'\in S_1(x')} \frac{p_{(x,x')(y,x')}^2 p_{(y,x')(y,y')}^2}{p_{(x,x')(y,x')} p_{(y,x')(y,y')} + p_{(x,x')(x,y')} p_{(x,y')(y,y')}  } \\
&= 2 \alpha^2 p_{xy}^2 + 3\alpha^2 p_{xy}p_{yx}- (\alpha \frac{d_{x}}{\mu_{x}}+ \beta \frac{d_{x'}}{\mu_{x'}})(\alpha p_{xy}) +
	3\alpha p_{xy}(\sum_{z\in S_2(x)} \alpha p_{yz}+ \sum_{y'\in S_1(x')} \beta p_{x'y'})\\
	&\phantom{=}+\alpha^2\sum_{\tilde{y}\in S_1(x)} (3p_{xy} p_{y\tilde{y}} + p_{x\tilde{y}} p_{\tilde{y}y})
	- 4\Bigl( \alpha^2\sum_{z\in S_2(x)} \frac{p_{xy}^2p_{yz}^2}{p_{xz}^{(2)}}+ \sum_{y'\in S_1(x')} \frac{(\alpha p_{xy})^2(\beta p_{x'y'})^2}{2\alpha\beta p_{xy} p_{x'y'}} \Bigr)\\
&= 4\alpha^2 Q(x)_{yy}.
\end{align*}

And similarly, $4Q((x,x'))_{(x,y')(x,y')} = 4\beta^2 Q(x)_{y'y'}$ for $y'\in S_1(x')$.

For $y_i\not=y_j\in S_1(x)$, we have from \eqref{eq:Q_offdiag} that
\begin{align*}
&4Q((x,x'))_{(y_i,x')(y_j,x')} \\
&= 2p_{(x,x')(y_i,x')}p_{(x,x')(y_j,x')} - 2p_{(x,x')(y_i,x')}p_{(y_i,x')(y_j,x')} - 2 p_{(x,x')(y_j,x')}p_{(y_j,x')(y_i,x')} \\
&\phantom{=} - 4 \sum_{z\in S_2(x)} \frac{ p_{(x,x')(y_i,x')} p_{(y_i,x')(z,x')} p_{(x,x')(y_j,x')} p_{(y_j,x')(z,x')} }{  \sum_{\tilde{y}\in S_1(x)} p_{(x,x')(\tilde{y},x')} p_{(\tilde{y},x')(z,x')}}\\
&=2\alpha^2 p_{xy_i}p_{xy_j} -  2\alpha^2 p_{xy_i}p_{y_iy_j} -  2\alpha^2 p_{xy_j}p_{y_jy_i} -4\sum_{z\in S_2(x)} \frac{ \alpha^4 p_{xy_i} p_{y_iz} p_{xy_j} p_{y_jz} }{  \sum_{\tilde{y}\in S_1(x)} \alpha^2 p_{x\tilde{y}} p_{\tilde{y}z}} \\
&= 4\alpha^2 Q(x)_{y_iy_j}.
\end{align*}
And similarly, $4Q((x,x'))_{(x,y'_i)(x,y'_j)} = 4\beta^2 Q(x)_{y'_iy'_j}$ for $y'_i\not=y'_j\in S_1(x')$.

For any $y\in S_1(x)$ and $y'\in S_1(x')$, we have from \eqref{eq:Q_offdiag} that
\begin{align*}
&4Q((x,x'))_{(y,x')(x,y')}\\
&= 2 p_{(x,x')(y,x')} p_{(x,x')(x,y')}- 4\frac{ p_{(x,x')(y,x')} p_{(y,x')(y,y')} p_{(x,x')(x,y')} p_{(x,y')(y,y')} }{ p_{(x,x')(y,x')} p_{(y,x')(y,y')} + p_{(x,x')(x,y')} p_{(x,y')(y,y')}} \\
&= 2 \alpha \beta p_{xy}  p_{x'y'} - 4\frac{(\alpha \beta p_{xy}  p_{x'y'})^2}{2\alpha \beta p_{xy}  p_{x'y'}}  = 0.
\end{align*}

We can conclude from the above calculation that $Q((x,x'))=\alpha^2 Q(x) \oplus \beta^2 Q(y)$. Note also that the matrix $\diag \mathbf{v}_0((x,x')) = \sqrt{\alpha} \diag\mathbf{v}_0(x) \oplus \sqrt{\beta} \diag \mathbf{v}_0(x')$. Therefore, we derive the curvature matrix as
\[
A_\infty((x,x')) = 2\diag\mathbf{v}_0((x,x'))^{-1} Q((x,x')) \diag\mathbf{v}_0((x,x')) ^{-1} = \alpha A_\infty(x) \oplus \beta A_\infty(x'),
\]
as desired.
\end{proof}

Next we prove Theorem \ref{thm:cartesian_curv_fct}, which will be rephrased in a more abstract way. This will be useful in the next section when we discuss the Ricci curvature of weighted manifolds in an analogous manner.

\begin{theorem}\label{thm:eigen_product}
For $i\in \{1,2\}$, let $A_i$ be $m_i\times m_i$ symmetric matrices and $\mathbf{v}_i$ be vectors in $ \IR^{m_i}$. Given fixed weights $\alpha,\beta>0$, let $A$ and $\mathbf{v}$ be given as
\[A=\alpha A_1\oplus \beta A_2\,\,\text{and}\,\,\mathbf{v}=\sqrt{\alpha}\mathbf{v}_1\oplus \sqrt{\beta}\mathbf{v}_2.\]
For $N\in (0,\infty]$, consider
\[A_i(N):=A_i-\frac{2}{N}\mathbf{v}_i\mathbf{v}_i^\top\,\,\text{and}\,\,A(N):=A-\frac{2}{N}\mathbf{v}\mathbf{v}^\top.\]
Then we have
\begin{equation} \label{eq:product_ineq_general}
\min \{ \alpha\lambda_1, \beta\lambda_2 \} \le \lambda_{\min{}}(A(N_1+N_2)) \le \max \{ \alpha\lambda_1, \beta\lambda_2 \},
\end{equation}
where $\lambda_i:=\lambda_{\min{}}(A_i(N_i))$.
\end{theorem}

\begin{proof}
%The proof is given in matrix form: for each $i\in \{1,2\}$, $A_i$ and $B_i=\v_i \v_i^\top$ are symmetric matrices of size $m_i$, and $\v_i$ is a vector.
Let us first consider the case $N_1,N_2\in (0,\infty)$. We have the matrix $\mathbf{v}\mathbf{v}^\top= \begin{pmatrix} \alpha \v_1 \v_1^\top & \sqrt{\alpha\beta}\v_1\v_2^\top \\ \sqrt{\alpha\beta}\v_2 \v_1^\top& \beta \v_2 \v_2^\top \end{pmatrix}$. It follows that
\begin{align*}
A(N_1+N_2) = \begin{pmatrix} \alpha A_1(N_1) & \\ & \beta A_2(N_2) \end{pmatrix}
+ \frac{2}{N_1+N_2} \underbrace{ \begin{pmatrix} \alpha \frac{N_2}{N_1} \v_1\v_1^\top & -\sqrt{\alpha\beta}\v_1\v_2^\top \\ -\sqrt{\alpha\beta}\v_2\v_1^\top& \beta \frac{N_1}{N_2}\v_2\v_2^\top \end{pmatrix} }_{=: J}.
\end{align*}
We want to verify that $J\succeq 0$, which will then imply the left inequality in \eqref{eq:product_ineq_general}.

For any vector $w=\begin{pmatrix} w_1 \\ w_2 \end{pmatrix}$ with $w_i\in \IR^{m_i}$, we have
\begin{align*}
w^\top J w
&= \alpha \frac{N_2}{N_1} w_1^\top \v_1\v_1^\top w_1+ \beta \frac{N_1}{N_2} w_2^\top \v_2\v_2^\top w_2 - 2 \sqrt{\alpha\beta}w_1^\top \v_1\v_2^\top w_2\\
&= \Bigl( \sqrt{\alpha \frac{N_2}{N_1}} w_1^\top \v_1 - \sqrt{\beta \frac{N_1}{N_2}} w_2^\top \v_2 \Bigr)^2 \ge 0.
\end{align*}
Thus $J\succeq 0$. Next we prove the right inequality in \eqref{eq:product_ineq_general}. For $i\in \{1,2\}$, we choose a unit eigenvector $\w_i$ such that $A_i \w_i = \lambda_i \w_i$ where $\lambda_i=\lambda_{\min{}}(A_i(N_i))$, and let $\w:=\begin{pmatrix} c_1\w_1 \\ c_2\w_2 \end{pmatrix}$ with arbitrary constants $c_i\not=0$. It follows from the Rayleigh quotient description that
%\begin{align*}
\begin{multline*}
\lambda_{\min{}}(A(N_1+N_2))
\le \frac{\w^\top \bigl(\alpha A_1(N_1) \oplus \beta A_2(N_2) \bigr) \w + \frac{2}{N_1+N_2} \w^\top J  \w}{\w^\top \w} \\
= \frac{1}{c_1^2+c_2^2} \left(\alpha c_1^2  \lambda_1+ \beta c_2^2 \lambda_2 + \frac{2}{N_1+N_2} \Bigl( c_1\sqrt{\alpha \frac{N_2}{N_1}} \w_1^\top \v_1 - c_2\sqrt{\beta \frac{N_1}{N_2}} \w_2^\top \v_2 \Bigr)^2  \right).
\end{multline*}
%\end{align*}
We may choose $c_1=\sqrt{\beta \frac{N_1}{N_2}} \w_2^\top {v_2}$ and $c_2=\sqrt{\alpha \frac{N_2}{N_1}} \w_1^\top {v_1}$ so that the square term above becomes zero. As a result,
\begin{align*}
\lambda_{\min{}}(A(N_1+N_2)) \le \frac{\alpha c_1^2  \lambda_1+ \beta c_2^2 \lambda_2 }{c_1^2+c_2^2} \le \max \{ \alpha \lambda_1, \beta \lambda_2\},
\end{align*}
which finishes the proof of \eqref{eq:product_ineq_general}.
The case that $N_1$ or $N_2$ equals $\infty$ is not hard to prove by modifying the above argument.
\end{proof}

Now, Theorem \ref{thm:cartesian_curv_fct} follows from Theorem \ref{thm:eigen_product} and the following general fact \cite[Proposition 7.3]{CLP20} about star product.

\begin{proposition}[\cite{CLP20}] Let $f_1,f_2:(0,\infty]\to \mathbb{R}$ be continuous monotone non-decreasing functions with $\lim_{N\to 0}f_i(N)=-\infty, i=1,2$. Then a function $F:(0,\infty]\to \mathbb{R}$ satisfies $F=f_1\ast f_2$ if and only if it holds for any $N_1,N_2\in (0,\infty)$ that
\[\min\{f_1(N_1),f_2(N_2)\}\leq F(N_1+N_2)\leq \max\{f_1(N_1), f_2(N_2)\},\]
and $F(\infty)=\lim_{N\to \infty}F(N)$.
\end{proposition}

%%%%%%%%%%%%%%%%%%%%%%%%%%%%%%%%%%%%%%%%%%%%%%%%%%%%%
\section{The case of weighted manifolds} \label{sect:weighted_manifolds}
%%%%%%%%%%%%%%%%%%%%%%%%%%%%%%%%%%%%%%%%%%%%%%%%%%%%%

In \cite[Section 1.6]{CLP20}, the authors briefly draw a comparison between the Bakry-\'Emery curvature functions of graphs and that of weighted Riemannian manifolds. Here, we discuss this comparison further by investigating the analogous result that the optimal lower Ricci curvature bound at a point on a weighted Riemannian manifold is also the minimal eigenvalue of a rank one perturbation of a curvature matrix.
%(compared to $\K_{G,x}(N)=\lambda_{\min} ( \mathcal{A}_N(x))= \lambda_{\min}\left(\mathcal{A}_\infty(x)-\frac{2}{N}\v_0\v_0^\top\right)$ in the case of graphs).

A \emph{weighted Riemannian manifold} is a triple
$(M^n,g,e^{-V}d\mathrm{vol}_g)$, where $(M^n,g)$ is an $n$-dimensional
Riemannian manifold, $d\mathrm{vol}_g$ is the Riemannian
volume element, and $V$ is a smooth real valued function on $M^n$. The\emph{ $N$-Bakry-\'Emery Ricci tensor }of
$(M^n,g,e^{-V}d\mathrm{vol}_g)$ is defined to be
\begin{equation}\label{eq:RicciTensor}
\Ric_{N,V}:=\mathrm{Ric}+\mathrm{Hess}\,V-\frac{\grad V\otimes \grad V}{N-n},
\end{equation}
where $\mathrm{Ric}$ is the Ricci curvature tensor of $(M^n,g)$, $\mathrm{Hess}\,V$ is the Hessian of $V$, and $\grad V$ is the gradient of $V$ (\cite{Bakry94,BE85}). Using the $V$-Laplacian $\Delta_V:=\Delta_g-g(\grad V,\grad\cdot)$, where $\Delta_g$ is the Laplace-Beltrami operator on $(M^n,g)$, one can define the Bakry-\'Emery curvature-dimension inequality $CD(\K,N)$ (at any point $x\in M$) as in Definition \ref{defn:BEcurvature}. Then $CD(\K,N),N\in (n,\infty]$ (at a given point $x\in M$) holds if and only if $\Ric_{N,V} \geq \K$ (at $x\in M$) (see \cite[pp. 93--94]{Bakry94}).

\begin{definition}\label{defn:BEcurvatureMfld}
Let $(M^n,g,e^{-V}d\mathrm{vol}_g)$ be a weighted Riemannian manifold. For a given $N\in (n,\infty]$, the Bakry-\'Emery curvature $\K(M,V,x;N)$ at a point $x\in M$  is defined to be the largest $\K$ such that $CD(\K,N)$ holds at $x$. The function $\K_{M,V,x}: (0,\infty]\to \IR$ given by $\K_{M,V,x}(N):=\K(M,V,x;N+n)$ is called the Bakry-\'Emery curvature function of $(M^n,g,e^{-V}d\mathrm{vol}_g)$ at $x$.
\end{definition}

\begin{remark}\label{remark:curvature_function_mfld}
Recall $n$ is the dimension of the underlying Riemannian manifold. The purpose to define the curvature function $\K_{M,V,x}$ on the interval $(0,\infty]$ instead of on $(n,\infty]$ is to make it compatible with the graph case.
\end{remark}

When $V$ is constant, that is, when the curvature-dimension inequality is based on the Laplace-Beltrami operator $\Delta_g$, the dimension parameter $N$ in \eqref{eq:RicciTensor} can be equal to $n$, and the function $\K_{M,V,x}$ is a constant function on $[0,\infty]$ (see \cite[pp. 93--94]{Bakry94}, \cite[Appendix C.6]{BGL}). When $V$ is not constant, $\K_{M,V,x}(N)$ tends to $-\infty$ as $N$ tends to $0$.

Next we investigate the shape of the Bakry-\'Emery curvature function $\K_{M,V,x}$ at $x$ on a weighted Riemannian manifold $(M^n,g,e^{-V}d\mathrm{vol}_g)$. If $\grad V(x)=0$, then this function $\K_{M,V,x}$ is constant. In the sequel, we consider the case that $\grad V(x)\neq 0$.
Since $\Ric_{N,V}$ is a symmetric $(0,2)$-tensor, there exists a linear transformation $\A_{N-n}: T_xM \to T_xM$ from the tangent space $T_xM$ of $M$ at $x$ to itself, such that
\[\Ric_{N,V}(v,v)=g(\A_{N-n} v, v),\,\,\text{for any }\, v\in T_xM.\]
Therefore, the optimal lower Ricci curvature bound at $x$ can be expressed as the minimal eigenvalue: \[\K_{M,V,x}(N):=\inf_{v\in S_xM} \Ric_{N+n,V}(v,v) = \lambda_{\min{}} (\A_N),\,\,\text{for any}\,\,N\in (0,\infty],\] where $S_xM$ stands for the space of unit tangent vectors at $x$. For any $v,w\in T_xM$, the tensor $\Ric_{N+n,V}, N\in(0,\infty]$ can be written independently of the choice of an orthonormal basis $\{e_i\}_{i=1}^n$ of the tangent space $T_xM$ as
\begin{align*}
\Ric_{N+n,V}(v,w)&=\mathrm{Ric}(v,w)+\mathrm{Hess}V(v,w)-\frac{v(V)\cdot w(V)}{N}\\
&=\sum_{i=1}^n g(R(v,e_i)e_i,w)+g(\nabla_v\grad V,w)-\frac{1}{N}g(g(\grad V,v)\grad V,w),
\end{align*}
where $\nabla_v \cdot$ is the covariant derivative along $v$, and $R(\cdot,\cdot)\cdot$ is the Riemann curvature tensor.
Let us define linear transformations $\A_\infty, \B: T_xM\to T_xM$ as follows: for any $v\in T_xM$,
\begin{align*}
\A_\infty v &:=\sum_{i=1}^nR(v,e_i)e_i+\nabla_v\grad V,\\
\B v&:=\frac{1}{2}g(\grad V,v)\grad V.
\end{align*}
Therefore, the linear transformation $\A_N: T_xM\to T_xM$ satisfies
%can be explicitly given as follows: for any $v\in T_xM$,
\begin{align} \label{eq:manifold}
\A_N = \A_\infty  - \frac{2}{N} \B.
\end{align}
Recall that $T_xM$ equipped with the inner product $g$ is an $n$-dimensional Euclidean vector space. Let $A_\infty, A_N$ be the matrix representation of $\A_\infty, \A_N$ with respect to an orthonormal basis $\{e_i\}_{i=1}^n$ of $T_xM$. Let $\mathbf{v}_0$ be the $n$-dimensional coordinate vector of $\frac{1}{\sqrt{2}}\grad V (x)$ with respect to $\{e_i\}_{i=1}^n$. Then we have a matrix version of \eqref{eq:manifold}:
\begin{equation}
A_N=A_\infty-\frac{2}{N}\mathbf{v}_0\mathbf{v}_0^\top.
\end{equation}
Notice that both $A_\infty$ and $A_N$ are symmetric $n\times n$ matrices. We call $A_\infty$ the \emph{the curvature matrix} at $x$ (with respect to the orthonormal basis $\{e_i\}_{i=1}^n$) of the weighted manifold $(M,g,e^{-V}d{\vol}_{g})$. The matrix $A_N$ is a rank one perturbation of the curvature matrix $A_\infty$. The Bakry-\'Emery curvature function satisfies
\begin{equation}
\K_{M,V,x}(N)=\lambda_{\min{}}(A_N)=\lambda_{\min{}}(A_\infty-\frac{2}{N}\v_0\v_0^\top).
\end{equation}

%where $\A_\infty$ represents \emph{the curvature matrix} at $x$ of the weighted manifold $(M,g,e^{-V}d{\vol}_{g})$, and  $\A_N=\A_\infty-\frac{1}{N-n}\B$ is a rank one perturbation of $\A_\infty$ with $\B := g(\grad V, \cdot) \grad V$. This result can be compared to the main result in the case of graphs: $\K_{G,x}(N)=\lambda_{\min{}}(A_N)=\lambda_{\min{}}(A_\infty-\frac{2}{N}\v_0\v_0^\top)$.

Therefore, we reduce the study of the Bakry-\'Emery curvature functions of weighted Riemannian manifolds to a matrix eigenvalue problem of the same type as in the graph case.

Then, it is direct to check the results (Theorem \ref{thm:continuous_and_threshold}, Proposition \ref{cor: N_threshold}, Theorem \ref{thm:lower_upper_bound}, and Propositions \ref{prop:curv_sharp_shape} and \ref{prop:v0_eigenvalue}) describing the shape of curvature functions of graphs also holds for the curvature functions of weighted manifolds.

In particular, we mention the quantity $\K_\infty^0(x)$ in the weighted manifold case
\begin{align*}
\K^0_\infty(x):=&\frac{\mathbf{v}_0^\top A_\infty\mathbf{v}_0}{\mathbf{v}_0^\top\mathbf{v}_0}=\frac{g(\A_\infty\grad V, \grad V)}{g(\grad V, \grad V)}\\
=&\Ric_x\left(\frac{\grad V}{\|\grad V\|}, \frac{\grad V}{\|\grad V\|}\right)+\frac{\grad V(x)}{\|\grad V(x)\|}\left(\|\grad V\|\right).
\end{align*}
Then we have
\begin{equation} \label{eq:upper_bound_manifold}
 \K_{M,V,x}(N) \le \frac{\mathbf{v}_0^\top A_N\mathbf{v}_0}{\mathbf{v}_0^\top\mathbf{v}_0}=\K^{0}_{\infty}(x) - \frac{1}{N}\|\grad V(x)\|^2.
\end{equation}
We say that $x$ is \emph{$N$-curvature sharp} if \eqref{eq:upper_bound_manifold} holds with equality. Then, for example, one can conclude similarly to Proposition \ref{prop:v0_eigenvalue} that $x$ is $\infty$-curvature sharp if and only if $\grad V(x)$ is an eigenvector corresponding to the minimal eigenvalue of $\A_\infty$.

Now we discuss the curvature functions of the Cartesian product of weighted Riemannian manifolds. Given two weighted manifolds $(M_i^{n_i}, g_i, e^{-V_i} d{\vol}_{g_i})$, $i \in \{1,2\}$, the Cartesian product $(M, g, e^{-V} d{\vol}_g) = (M_1\times M_2, g_1\oplus g_2, e^{-V_1\oplus V_2} d{\vol}_{g_1\oplus g_2})$ has a canonical identification of the tangent space $T_{(x_1,x_2)}M \simeq T_{x_1}M_1 \oplus T_{x_2}M_2$. We observe that $\A_\infty, \B$ of the product is naturally decomposed into the corresponding $\A_\infty,\B$ in each factor, that is,
\begin{align*}
\A_\infty^{M}(v_1 \oplus v_2)
&= \sum_{i=1}^{n_1} R(v_1,e_i)e_i \oplus \sum_{j=1}^{n_2} R(v_2,e_i)e_i+ \nabla_{v_1}\grad V_1\oplus \nabla_{v_2}\grad V_2  \\
&= \A_\infty^{M_1}(v_1) \oplus \A_\infty^{M_2}(v_2),
\end{align*}
and $\B(v_1 \oplus v_2) = g(\grad V_1 \oplus \grad V_2, v_1\oplus v_2 ) (\grad V_1 \oplus \grad V_2)$, for any $v_i\in T_{x_i}M_i$.
In the matrix form, we have
\begin{align*}
A_\infty ^M=A_\infty^{M_1}\oplus A_\infty^{M_2},\,\,\text{and}\,\,\mathbf{v}_0^M=\mathbf{v}_0^{M_1}\oplus\mathbf{v}_0^{M_2}.
\end{align*}

Theorem \ref{thm:eigen_product} is then applicable for manifolds and yields the following theorem.
\begin{theorem} \label{thm:cartesian_curv_fct_mfld}
The curvature function of the Cartesian product \[(M, g, e^{-V} d{\vol}_g) = (M_1\times M_2, g_1\oplus g_2, e^{-V_1\oplus V_2} d{\vol}_{g_1\oplus g_2})\] satisfies the following inequalities:
\[ \min \{ \lambda_{\min{}}(\A_{N_1}^{M_1}), \lambda_{\min{}}(\A_{N_2}^{M_2}) \} \le  \lambda_{\min{}}(\A_{N_1+N_2}^M) \le \max \{ \lambda_{\min{}}(\A_{N_1}^{M_1}), \lambda_{\min{}}(\A_{N_2}^{M_2}) \}, \]
and consequently, $\K_{M,V,(x_1,x_2)} = \K_{M_1, V_1, x_1} \ast  \K_{M_2, V_2, x_2}$.
\end{theorem}

We conclude this section with the following example of a weighted Riemannian manifold with $\infty$-curvature sharp points.
\begin{example} [weighted $2$-sphere] \label{ex:2-sphere}

Let $M = S^2(r)$ be the two-dimensional sphere of radius $r$ with coordinates $\x(\theta,\phi)=( r\cos \theta \cos \phi, r\sin \theta \cos \phi, r\sin \phi)$ for $\theta \in (0,2\pi)$ and $\phi\in (-\pi/2, \pi/2)$, and the corresponding metric $g=r^2(\cos^2(\phi) d\theta\otimes d\theta+d\phi\otimes d\phi)$. Let $V:M\to \IR$ be a smooth height function, i.e., $V(\x(\theta,\phi))=h(\phi)$ for some smooth function $h$. Consider the weighted manifold $(M, g, e^{-V} d{\vol}_g)$ and a point $p\in M$ with $\grad V(p)\neq 0$.

The tangent space $T_pM$ is spanned by $\grad V(p)$ and the tangent vector $\x_{\theta}(p) \in T_pM$, which satisfy $g(\x_{\theta}(p),\grad V(p))=0$. We first check that both the tangent vector $\grad V(p)$ and $\x_{\theta}(p)$ are in fact eigenvectors of $\A_N$ (based at $p$).

Consider the geodesic $\alpha:=\x(\theta_0,\cdot):(-\pi/2,\pi/2) \to S^2(r)$ which passes through $p=\alpha(t)=\x(\theta_0,t)$. The tangent vector $\grad V(p)$ is parallel to $\alpha'(t)$, so we may write $\grad V(p)=k(t) \alpha'(t)$ for some smooth function $k:(-\pi/2,\pi/2) \to \IR$. It follows that at $p\in M$
\begin{align*}
\nabla_{\grad V} \grad V &= \nabla_{k\alpha'}k\alpha' = k\nabla_{\alpha'}k\alpha' = k^2\underbrace{\nabla_{\alpha'}\alpha'}_{=0}+ kk'\alpha'=k'\grad V,
\end{align*}
and hence,
\begin{align*}
\A_N \grad V
&= \sum_{i=1}^2 R(\grad V,e_i)e_i+\nabla_{\grad V }\grad V - \frac{g(\grad V , \grad V )}{N}\grad V \\
&= \left( \frac{1}{r^2} + k' - \frac{k^2 r^2}{N} \right) \grad V.
\end{align*}
%so $\grad V$ is indeed an eigenvalue of $\A_N$
On the other hand, we have $g_p(\nabla_{\x_{\theta}} \grad V, \grad V)=0$, and
\begin{align*}
g_p\left(\nabla_{\x_{\theta}}\grad V, \frac{\x_{\theta}}{\|\x_{\theta}\|} \right)=-\|x_{\theta}\| g_p \left( \grad V, \nabla_{\frac{\x_{\theta}}{\|\x_{\theta}\|}} \frac{\x_{\theta}}{\|\x_{\theta}\|} \right) = - \|x_{\theta}\|g_p\left(\grad V, k_g \frac{\x_{\phi}}{\|\x_{\phi}\|}\right),
\end{align*}
where $k_g=\frac{1}{r} \tan(t)$ is the geodesic curvature of the parallel circles with the unit tangent vector $\frac{\x_{\theta}}{\|\x_{\theta}\|}(p)$, and $\frac{\x_{\phi}}{\|\x_{\phi}\|}(p)=\frac{\alpha'(t)}{\|\alpha'(t)\|}=\frac{\alpha'(t)}{r}$. It follows that
\begin{align*}
\nabla_{\x_{\theta}(p)} \grad V = - g \left(\grad V(p), \frac{\alpha'(t)}{r} \right) \frac{1}{r}\tan(t)\x_{\theta}(p) = - k(t) \tan(t) \x_{\theta}(p).
\end{align*}
Therefore, we have
\begin{align*}
\A_N \x_{\theta}(p) &= \sum_{i=1}^2 R(\x_{\theta}, e_i)e_i+\nabla_{\x_{\theta} }\grad V - \frac{g(\grad V(p), \x_{\theta} )}{N}\grad V(p) \\
&= \left( \frac{1}{r^2} - k(t) \tan(t) \right) \x_{\theta}(p).
\end{align*}
It means that both $\grad V(p)$ and $\x_{\theta}(p)$ are eigenvectors of $\A_N$. Then the Bakry-\'Emery curvature and the generalised scalar curvature at $p$ are given by
\begin{align*}
\K_{M,V,p}(N) &= \min \left\{ \frac{1}{r^2} + k'(t) - \frac{k(t)^2 r^2}{N} , \frac{1}{r^2} - k(t) \tan(t) \right\}, \\
S_{M,V,p}(N) &= \frac{2}{r^2} - k(t)\tan(t)+ k'(t) - \frac{k(t)^2 r^2}{N}.
\end{align*}

Recall that the point $p=\alpha(t)$ is $\infty$-curvature sharp if and only if $\grad V(p)$ corresponds to the minimal eigenvalue of $\A_\infty$, which occurs precisely when $k'(t) \le - \tan(t) k(t)$.
In the special case when $k:(-\pi/2,\pi/2) \to \IR$ is  even, either the point $p=\alpha(t)$ or its mirror $p'=\alpha(-t)$ (or both) is $\infty$-curvature sharp. In particular, if $k(\cdot)=c\cos(\cdot)$ on the whole interval $(-\pi/2,\pi/2)$ for some $c\neq 0$ (which means $V(x,y,z)=az+b$ with $a=cr$ and $b\in \mathbb{R}$), then we have $k' = - \tan(\cdot) k$ and hence $p$ is $\infty$-curvature sharp. In fact, every point of $M$ except for the south and north poles (i.e., when $\phi=-\pi/2, \pi/2$) is $\infty$-curvature sharp. At the south and north poles, the curvature functions are constant. Moreover, this choice of $k$ provides a non-constant potential function $V$ for the round sphere as a gradient Ricci soliton.
\end{example}

A complete Riemannian manifold $(M,g)$ is called a gradient Ricci soliton with a potential function $V$ if $\Ric_{V,\infty} = \lambda g$ for some constant $\lambda$ (see, e.g., \cite[Definition 1.2.3]{Topping}).
\begin{theorem} \label{thm:soliton}
Every gradient Ricci soliton $(M,g)$ with a potential function $V$ leads to a weighted Riemannian manifold $(M,g,e^{-V}d{\vol}_g)$ which is $\infty$-curvature sharp at every point $x\in M$ with $\grad V (x)\not=0$.
\end{theorem}
\begin{proof}
Being a gradient Ricci soliton means, for every point $x\in M$, $\A_\infty(x)=\lambda {\rm Id}$. In particular, if $\grad V (x)$ is nonzero, then it is an eigenvector corresponding to the smallest eigenvalue of $\A_\infty(x)$, and therefore $x$ is $\infty$-curvature sharp by Theorem \ref{prop:v0_eigenvalue}(ii).
\end{proof}

%%%%%%%%%%%%%%%%%%%%%%%%%%%%%%%%%%%%%%%%%%%%%%%%%%%%%
\section{Geometric structure of $B_2(x)$ and curvature properties} \label{sect:geom_property}
%%%%%%%%%%%%%%%%%%%%%%%%%%%%%%%%%%%%%%%%%%%%%%%%%%%%%

In this section, we present the proofs of the three results (Proposition \ref{prop:scalar_curv}, and Theorems \ref{thm:curv_sharp_criterion} and \ref{thm:graph_modify_conjecture}) about the curvature at $x$ which are related to the geometric structure of the $B_2(x)$.

\begin{proof}[Proof of Propositions \ref{prop:scalar_curv}]
The computations for the matrix $A_\infty$ and \[S_{G,x}(N)=\trace(A_\infty - \frac{2}{N}\v_0\v_0^\top)=S_{G,x}(\infty)-\frac{2}{N}\frac{d_x}{\mu_x}\] are given in Appendix \eqref{eq:curMatrixLaplacian} and \eqref{eq:scalar_compute}. In the particular case of non-weighted graphs, the terms in \eqref{eq:curMatrixLaplacian} and \eqref{eq:scalar_compute} are simplified by $\mu_x=1$ and $p_{uv}\in \{0,1\}$, which directly gives the desired result.
%\[ S_{G,x}(N) = \frac{1}{2}(5-d_x)d_x+ \sum_{y\in S_1(x)} \frac{3}{2}(d^+(y) + 2 d^0(y)) - 2|S_2(x)| - \frac{2}{N}d_x. \]
\end{proof}

\begin{proof}[Proof of Theorem \ref{thm:curv_sharp_criterion}]
In view of Proposition \ref{prop:v0_eigenvalue}, we need to show that $\v_0$ is an eigenvector of $A_\infty$ under the $S_1$-in and $S_1$-out regularity assumption: $p^{-}(y):=p_{yx}$ and $p^{+}(y):=\sum_{z\in S_2(x)} p_{yz}$ are independent of $y\in S_1(x)$.

The vector $\v_0=\begin{pmatrix}\sqrt{p_{xy_1}} & \sqrt{p_{xy_2}}  & \cdots & \sqrt{p_{xy_m}} \end{pmatrix}^\top$ is an eigenvector of $A_\infty$ if and only if $\lambda \v_0 = A_\infty \v_0 = 2\diag(\v_0)^{-1} Q \diag(\v_0)^{-1}$ for some $\lambda\in \IR$, or equivalently,
\[2Q \mathbbm{1}_m=\lambda \begin{pmatrix} p_{xy_1} & p_{xy_2}  & \cdots & p_{xy_m} \end{pmatrix}^\top,\]
that is, $\frac{1}{p_{xy_i}}\sum_{j=1}^{m} Q_{y_i y_j} = \frac{1}{2}\lambda$ is independent of $i\in [m]:=\{1,2,\ldots,m\}$.

A direct calculation using the formula \eqref{eq:Gamma2S1S1Laplacian} yields, for any $i\in [m]$,
\begin{align*}
\frac{1}{p_{xy_i}}\sum_{j=1}^{m} Q_{y_i y_j}
&= \frac{1}{4}\frac{d_x}{\mu_x} +\frac{3}{4} p_{y_ix}- \frac{1}{4}\sum_{z\in S_2(x)} p_{y_iz} +
\sum_{j=1}^{m} (\frac{1}{4} p_{y_i y_j} -\frac{1}{4}\frac{p_{xy_j}p_{y_jy_i}}{p_{xy_i}})  \\
&= \frac{1}{4}\frac{d_x}{\mu_x} +\frac{3}{4} p_{y_ix} - \frac{1}{4}\sum_{z\in S_2(x)} p_{y_iz} +\frac{1}{4} \sum_{j=1}^{m} p_{y_i y_j} (\underbrace{1-\frac{p_{y_jx}}{p_{y_ix}}}_{=0}),
\end{align*}
which is independent of $i$, given that $x$ is $S_1$-in and $S_1$-out regular.
\end{proof}

\begin{proof}[Proof of Theorem \ref{thm:graph_modify_conjecture}]
We denote by $\widetilde{Q}$, $\widetilde{A}_\infty$ and $\widetilde{A}_N$ the corresponding matrices $Q$, $A_\infty$ and $A_N$ centered at the vertex $x$ of the modified graph $\widetilde{G}=(V,\tilde{w}, \mu)$. We aim to prove that $\K_{\widetilde{G},x}(N) \ge \K_{G,x}(N)$, that is, $\lambda_{\min{}}(\widetilde{A}_N) \ge \lambda_{\min{}}(A_N)$. It suffices to show that $\widetilde{A}_N-A_N$ is positive semidefinite, since it would then imply that $\lambda_{\min{}}(\widetilde{A}_N)\ge  \lambda_{\min{}}(\widetilde{A}_N-A_N)+ \lambda_{\min{}}(A_N) \ge \lambda_{\min{}}(A_N)$.

Note that the vector $\mathbf{v}_0 = (\sqrt{p_{xy_1}} \ \sqrt{p_{xy_2}} \ ... \ \sqrt{p_{xy_m}})^\top$ is unchanged under this graph modification, so we have
$\widetilde{A}_N-A_N = 2\diag(\mathbf{v}_0)^{-1}(\widetilde{Q}-Q)\diag(\mathbf{v}_0)^{-1}$. To prove that $\widetilde{A}_N-A_N \succeq 0$ is equivalent to showing that $\widetilde{Q}-Q \succeq 0$.

\textbf{Operation (O1):}
The modification $\tilde{w}_{yy'}=w_{yy'}+C_1$ for a constant $C_1>0$ means $\tilde{p}_{yy'} - p_{yy'} = \frac{C_1}{\mu_y}$ and $\tilde{p}_{y'y} - p_{y'y} = \frac{C_1}{\mu_{y'}}$. We then derive from the formulae \eqref{eq:Q_ondiag} and \eqref{eq:Q_offdiag} that the matrix $\widetilde{Q}-Q$ have four nontrivial entries:
\begin{align*}
(\widetilde{Q}-Q)_{yy} &= \frac{1}{4}\left(3p_{xy}(\tilde{p}_{yy'} - p_{yy'})+ p_{xy'}(\tilde{p}_{y'y} - p_{y'y})\right) \\
&= \frac{C_1}{4}\left(3\frac{p_{xy}}{\mu_{y}}+ \frac{p_{xy'}}{\mu_{y'}}\right) = \frac{C_1}{4\mu_x}(3p_{yx}+ p_{y'x}),
\end{align*}
and similarly, $(\widetilde{Q}-Q)_{y'y'}=\frac{C_1}{4\mu_x}(p_{yx}+ 3p_{y'x})$ and $(\widetilde{Q}-Q)_{yy'}=(\widetilde{Q}-Q)_{y'y}=-\frac{C_1}{2\mu_x}(p_{yx}+ p_{y'x})$.

Consequently, the matrix $\widetilde{Q}-Q$ has two nontrivial eigenvalues, corresponding to those of the following $2\times 2$ matrix
\begin{align*}
	\left(\begin{array}{cc}
		  (\widetilde{Q}-Q)_{yy} & (\widetilde{Q}-Q)_{yy'} \\
		  (\widetilde{Q}-Q)_{y'y} & (\widetilde{Q}-Q)_{y'y'}
	\end{array}
	\right) =
	\frac{C}{4\mu_x}\left(\begin{array}{cc}
		  3p_{yx} + p_{y'x} & -2p_{yx} -2p_{y'x} \\
		  -2p_{yx} -2p_{y'x}& p_{yx}+ 3p_{y'x}
	\end{array}
	\right).
\end{align*}
This matrix has eigenvalues $p_{yx} + p_{y'x} \pm \sqrt{2p_{yx}^2+2p_{y'x}^2}$, and it becomes positive semidefinite when we assume $p_{yx} = p_{y'x}$.

\textbf{Operation (O2):}
Note that the edge-weight modification $\tilde{w}_{yy'}=w_{yy'}+C_2 w_{yz_0}w_{z_0y'}$ for all different $y,y'\in S_1(x)$ means
$\tilde{p}_{yy'} - p_{yy'}=C_2p_{yz_0}p_{z_0y'}\mu_{z_0}$.

For $y_i,y_j\in S_1(x)$ such that $y_i\neq y_j$, the formula \eqref{eq:Q_offdiag} gives
\begin{align} \label{eq:Q_diff_offdiag}
(\widetilde{Q}-Q)_{y_i y_j}
&= -\frac{1}{2}p_{xy_i}(\tilde{p}_{y_iy_j} - p_{y_iy_j} ) -\frac{1}{2}p_{xy_j}(\tilde{p}_{y_jy_i} - p_{y_jy_i} ) + \frac{p_{xy_i} p_{y_iz_0}p_{xy_j} p_{y_jz_0}}{p_{xz_0}^{(2)}} \\
&= -\frac{1}{2}p_{xy_i}\cdot C_2p_{y_iz_0} p_{z_0y_j}\mu_{z_0} -\frac{1}{2}p_{xy_j}\cdot C_2p_{y_jz_0} p_{z_0y_i}\mu_{z_0} + \frac{p_{xy_i} p_{y_iz_0}p_{xy_j} p_{y_jz_0}}{p_{xz_0}^{(2)}} \nonumber\\
&= -C_2p_{xy_i}p_{y_iz_0}p_{z_0y_j}\mu_{z_0} + \frac{p_{xy_i} p_{y_iz_0}p_{xy_j} p_{y_jz_0}}{p_{xz_0}^{(2)}} \nonumber\\
&= - p_{xy_i}p_{y_iz_0}p_{z_0y_j} \left(C_2\mu_{z_0}-\frac{p_{xy_j} p_{y_jz_0}}{p_{z_0y_j}p_{xz_0}^{(2)}}\right) \nonumber,
\end{align}
where the third equation is due to $p_{xy_i}p_{y_iz_0}p_{z_0y_j} = p_{xy_j}p_{y_jz_0}p_{z_0y_i}$ which can be checked by
\begin{align*}
\frac{p_{xy_i}p_{y_iz_0}p_{z_0y_j}}{p_{xy_j}p_{y_jz_0}p_{z_0y_i}}
= \frac{w_{xy_i}w_{y_iz_0} w_{z_0y_j}}{\mu_x\mu_{y_i}\mu_{z_0}} \cdot \frac{\mu_x\mu_{y_j}\mu_{z_0}}{w_{xy_j}w_{y_jz_0} w_{z_0y_i}}=\frac{w_{xy_i} \mu_{y_j}}{\mu_{y_i}w_{xy_j}}=\frac{p_{y_ix}}{p_{y_jx}}=\frac{p^{-}(y)}{p^{-}(y)}=1.
\end{align*}

For $y_i\in S_1(x)$, the formula \eqref{eq:Q_ondiag} gives
\begin{align} \label{eq:Q_diff_ondiag}
(\widetilde{Q}-Q)_{y_i y_i}
&= -\frac{3}{4}p_{xy_i}p_{y_iz_0}+\frac{1}{4}\sum_{y_j\not=y_i}\left(3p_{xy_i}(\tilde{p}_{y_iy_j} - p_{y_iy_j} ) + p_{xy_j}(\tilde{p}_{y_jy_i}-p_{y_jy_i})\right)+\frac{p_{xy_i}^2 p_{y_iz_0}^2}{p_{xz_0}^{(2)}} \\
&= -\frac{3}{4}p_{xy_i}p_{y_iz_0}+\sum_{y_j\not=y_i}\left( \frac{3C_2\mu_{z_0}}{4}p_{xy_i} p_{y_iz_0} p_{z_0y_j}+\frac{C_2\mu_{z_0}}{4}p_{xy_j}p_{y_jz_0} p_{z_0y_i}\right)+\frac{p_{xy_i}^2 p_{y_iz_0}^2}{p_{xz_0}^{(2)}} \nonumber\\
&= -\frac{3}{4}p_{xy_i}p_{y_iz_0} + C_2\mu_{z_0}\sum_{y_j\not=y_i}p_{xy_i} p_{y_iz_0} p_{z_0y_j} +\frac{p_{xy_i}^2 p_{y_iz_0}^2}{p_{xz_0}^{(2)}} \nonumber\\
&=p_{xy_i}p_{y_iz_0} \left(-\frac{3}{4}+C_2\mu_{z_0}\sum_{y_j\not=y_i}p_{z_0y_j} + \frac{p_{xy_i} p_{y_iz_0}}{p_{xz_0}^{(2)}}\right). \nonumber
\end{align}
Combining \eqref{eq:Q_diff_ondiag} and \eqref{eq:Q_diff_offdiag}, we derive the sum of entries in $i$-th row as
\begin{align*}
(\widetilde{Q}-Q)_{y_i y_i}+\sum_{j\not=i}(\widetilde{Q}-Q)_{y_i y_j}
&= p_{xy_i}p_{y_iz_0} \left( -\frac{3}{4} + \frac{p_{xy_i} p_{y_iz_0}}{p_{xz_0}^{(2)}} + \sum_{j\not=i}\frac{p_{xy_j} p_{y_jz_0}}{p_{xz_0}^{(2)}}\right)\\
&= p_{xy_i}p_{y_iz_0} \left( -\frac{3}{4} +\frac{1}{p_{xz_0}^{(2)}}\sum_{y\in S_1(x)} p_{xy} p_{yz_0} \right)\\
&= \frac{1}{4}p_{xy_i}p_{y_iz_0} >0.
\end{align*}
(Note that the terms involving $C_2$ are cancelled out in the above expression.)

Under the assumption that $\displaystyle C_2\mu_{z_0} \ge \frac{p_{xy_j} p_{y_jz_0}}{p_{z_0y_j}p_{xz_0}^{(2)}}$ for all $j\not=i$, we can guarantee in \eqref{eq:Q_diff_offdiag} that $(\widetilde{Q}-Q)_{y_i y_j}\le 0$. It then follows that
$$(\widetilde{Q}-Q)_{y_i y_i}> - \sum_{j\not=i}(\widetilde{Q}-Q)_{y_i y_j} = \sum_{j\not=i}\left|(\widetilde{Q}-Q)_{y_i y_j}\right|,$$ which shows $\widetilde{Q}-Q$ is diagonally dominant and hence $\widetilde{Q}-Q\succeq 0$.

Finally, we remark that the assumption $\displaystyle C_2\mu_{z_0} \ge \frac{p_{xy_j} p_{y_jz_0}}{p_{z_0y_j}p_{xz_0}^{(2)}}$ can be re-written as the assumption given in Theorem \ref{thm:graph_modify_conjecture}, namely $\displaystyle C_2 \ge \frac{p^{-}(y)}{\mu_x p_{xz_0}^{(2)}}$ due to the following identity:
\begin{align*}
\frac{p_{xy_j} p_{y_jz_0}}{p_{z_0y_j}p_{xz_0}^{(2)}}\cdot\frac{1}{\mu_{z_0}}
= \frac{w_{xy_j} w_{y_jz_0}}{\mu_x\mu_{y_j}w_{z_0y_j}p_{xz_0}^{(2)}}=\frac{p_{y_jx}}{\mu_x p_{xz_0}^{(2)}}=\frac{p^{-}(y)}{\mu_x p_{xz_0}^{(2)}}.
\end{align*}
\end{proof}
\appendix
%%%%%%%%%%%%%%%%%%%%%%%%%%%%%%%%%%%%%%%%%%%%%%%%%%%%%%
\section{Explicit Structure of relevant matrices} \label{sect:appendix}

In this section we collect the explicit expressions of matrices $\left(\Delta(x)\Delta(x)^\top\right)_{\hat{1}}$, $\Gamma(x)_{\hat{1}}$, $\Gamma_2(x)_{\hat{1}}$, $Q(x)$ and $A_\infty(x)$, all of which are important ingredients to our curvature calculation in Theorem \ref{thm:eigenvalue_main}. These expressions will be given in \eqref{eq:DeltaDeltaT}, \eqref{eq:Gamma=diag}, \eqref{eq:Gamma_2_yy}-\eqref{eq:Gamma_2_zz_off}, \eqref{eq:Q_ondiag}-\eqref{eq:Q_offdiag} and \eqref{eq:A_infty_entries}, respectively.

We fix the central vertex $x\in V$ and let $m=|S_1(x)|$ and $n=|S_2(x)|$ be the size of $1$-sphere and $2$-sphere around $x$, respectively. The vertices in $S_1(x)$ and $S_2(x)$ are indexed by
\begin{align*}
S_1(x) = \{y_1, y_2,....,y_m\}; \qquad
S_2(x) = \{z_1, z_2,....,z_n\}.
\end{align*}

The linear operator $\Delta(\cdot)(x)$ and the bilinear forms $\Gamma(\cdot, \cdot)(x), \Gamma_2(\cdot, \cdot)(x)$ can be represented by a vector $\Delta(x)$ and matrices $\Gamma(x), \Gamma_2(x)$ as follows:
\begin{align*}
	\Delta f(x) &= \Delta(x)^\top \vec{f},\\
	\Gamma(f,g)(x) &=\vec{f}^\top \Gamma(x) \vec{g}, \\
	\Gamma_2(f,g)(x) &=\vec{f}^\top \Gamma_2(x) \vec{g}.
\end{align*}
In the first two equations, $\vec{f}$ and $\vec{g}$ are vector representations indexed by vertices in $B_1(x)$ as
\begin{equation*}
	\vec{f}=\left(
	\begin{array}{cccc}
		f(x) & f(y_1) & \cdots & f(y_m) \\
	\end{array}
	\right)^\top,
\end{equation*}
and similarly for $\vec{g}$. In the last equation, $\vec{f}$ and $\vec{g}$ are vector representations indexed by vertices in $B_2(x)$ as
\begin{equation*}
	\vec{f}=\left(
	\begin{array}{ccccccc}
		f(x) & f(y_1) & \cdots & f(y_m) & f(z_1) & \cdots & f(z_n) \\
	\end{array}
	\right)^\top,
\end{equation*}
and similarly for $\vec{g}$.

More explicitly, the defining equation $\Delta f(x) = \sum_{y\in S_1(x)} p_{xy} (f(y)-f(x))$ is translated to
\begin{equation*}
	\Delta(x)=\left(
	\begin{array}{ccccc}
		-\frac{d_x}{\mu_x} & p_{xy_1} & p_{xy_2} & \cdots & p_{xy_m} \\
	\end{array}
	\right)^\top,
\end{equation*}
and $\Delta(x)_{S_1}=\left(
	\begin{array}{cccc}
		p_{xy_1} & p_{xy_2} & \cdots & p_{xy_m} \\
	\end{array}
	\right)^\top$
when restricted to the vertices in $S_1(x)$.

Hence we derive that
\begin{equation}\label{eq:DeltaDeltaT}
\left(\Delta(x)\Delta(x)^\top\right)_{\hat{1}}=\Delta(x)_{S_1}\Delta(x)_{S_1}^\top=\left(
                                                                          \begin{array}{cccc}
                                                                            p_{xy_1}^2 & p_{xy_1} p_{xy_2} & \cdots &  p_{xy_1} p_{xy_m} \\
                                                                            p_{xy_2}p_{xy_1} &  p_{xy_2}^2 & \cdots &  p_{xy_2} p_{xy_m} \\
                                                                            \vdots & \vdots & \ddots & \vdots \\
                                                                             p_{xy_m} p_{xy_1} &  p_{xy_m} p_{xy_2} & \cdots &  p_{xy_m}^2 \\
                                                                          \end{array}
                                                                        \right).
\end{equation}

The defining equation $2\Gamma(f,g) = \Delta(f\cdot g) - f\cdot \Delta g - \Delta f \cdot g$ means
\begin{align*}
2\Gamma (f,g)(x)
&= \sum_{y\in S_1(x)}  p_{xy}[(f(y)g(y)-f(x)g(x))- f(x)(g(y)-g(x)) - (f(y)-f(x))g(x)] \\
&= \sum_{y\in S_1(x)}  p_{xy}(f(y)-f(x))(g(y)-g(x)) \\
&= \sum_{y\in S_1(x)} p_{xy}[f(x)g(x)-f(x)g(y)-f(y)g(x)+f(y)g(y)],
\end{align*}
which can be translated to
\begin{equation*}
	\Gamma(x)=\frac{1}{2}\left(
	\begin{array}{ccccc}
		\frac{d_x}{\mu_x} & -p_{xy_1} & -p_{xy_2} & \cdots & -p_{xy_m} \\
		-p_{xy_1} & p_{xy_1} & & & \\
		-p_{xy_2} & & p_{xy_2} & & \\
		\vdots & & & \ddots & \\
		-p_{xy_m} & & & & p_{xy_m}
	\end{array}
	\right).
\end{equation*} In particular, after removing the first row and column corresponding to the vertex $x$, we simply have
\begin{equation} \label{eq:Gamma=diag}
	\Gamma(x)_{\hat{1}}=\Gamma(x)_{S_1,S_1}=\frac{1}{2}\mathrm{diag}\left(
	\begin{array}{cccc}
		p_{xy_1} & p_{xy_2} & \cdots & p_{xy_m} \\
	\end{array}
	\right)= \frac{1}{2}\mathrm{diag}\left(\Delta(x)_{S_1}\right).
\end{equation}

Now we discuss the structure of the matrix $\Gamma_2(x)$. After removing the first row and column corresponding to $x$, the matrix $\Gamma_2(x)_{\hat{1}}$ has the following block structure in $S_1(x)\sqcup S_2(x)$:
\begin{equation} \label{eq:Gamma_2_1hat}
	\Gamma_2(x)_{\hat{1}}=\Gamma_2(x)_{S_1\cup S_2, S_1\cup S_2}=\left(
	\begin{array}{cc}
		\Gamma_2(x)_{S_1,S_1} & \Gamma_2(x)_{S_1,S_2} \\
		\Gamma_2(x)_{S_2,S_1} & \Gamma_2(x)_{S_2,S_2} \\
	\end{array}
	\right).
\end{equation}
%A lengthy computation for $\Gamma_2(x)$ via the defining equation $2\Gamma_2(f,g) = \Delta(\Gamma(f,g)) - \Gamma(f,\Delta g) - \Gamma(g,\Delta f)$ yields the expressions \eqref{eq:Gamma_2_yy} - \eqref{eq:Gamma_2_zz} below.
The defining equation $2\Gamma_2(f,g) = \Delta(\Gamma(f,g)) - \Gamma(f,\Delta g) - \Gamma(g,\Delta f)$ yields
\begin{align*}
4\Gamma_2(f,g)(x)=&\sum_{y\in S_1(x)}p_{xy}\left[2\Gamma(f,g)(y)-(f(y)-f(x))\Delta g(y)-(g(y)-g(x))\Delta f(y)\right]\\
&-\frac{2d_x}{\mu_x}\Gamma(f,g)(x)+2\Delta f(x)\Delta g(x).
\end{align*}
Let us denote by $I(\cdot,\cdot)(x)$ the bilinear form defined via
\[I(f,g)(x):=\sum_{y\in S_1(x)}p_{xy}\left[2\Gamma(f,g)(y)-(f(y)-f(x))\Delta g(y)-(g(y)-g(x))\Delta f(y)\right],\]
and the matrix representing it by $I(x)$. Then we have
\begin{equation*}
4\Gamma_2(x)=I(x)-\frac{2d_x}{\mu_x}\Gamma(x)+2\Delta(x)\Delta(x)^\top,
\end{equation*}
and hence
\begin{equation}\label{eq:Gamma2GammaDelta}
4\Gamma_2(x)_{\hat{1}}=I(x)_{\hat{1}}-\frac{2d_x}{\mu_x}\Gamma(x)_{\hat{1}}+2\left(\Delta(x)\Delta(x)^\top\right)_{\hat{1}}.
\end{equation}
In order to derive $I(x)_{\hat{1}}$, we only need to compute the expression of $I(f,g)(x)$ for functions $f$ and $g$ satisfying $f(x)=g(x)=0$:
\begin{align*}
&I(f,g)(x)\\
=&\sum_{y\in S_1(x)}p_{xy}\left[2\Gamma(f,g)(y)-f(y)\Delta g(y)-g(y)\Delta f(y)\right]\\
=&\sum_{y\in S_1(x)}\sum_{z\in S_1(y)}p_{xy}p_{yz}\left[(f(z)-f(y))(g(z)-g(y))-f(y)(g(z)-g(y))-g(y)(f(z)-f(y))\right]\\
=&\sum_{y\in S_1(x)}\sum_{z\in S_1(y)}p_{xy}p_{yz}\left[f(z)g(z)-2f(z)g(y)-2f(y)g(z)+3f(y)g(y)\right]\\
=&\sum_{y\in S_1(x)}\sum_{z\in S_2(x)}p_{xy}p_{yz}\left[f(z)g(z)-2f(z)g(y)-2f(y)g(z)+3f(y)g(y)\right]\\
& +\sum_{y\in S_1(x)}\sum_{y'\in S_1(x)}p_{xy}p_{yy'}\left[f(y')g(y')-2f(y')g(y)-2f(y)g(y')+3f(y)g(y)\right]\\
& +\sum_{y\in S_1(x)}p_{xy}p_{yx}\cdot 3f(y)g(y)\\
=&\sum_{y\in S_1(x)}\left[3p_{xy}p_{yx}+3p_{xy}\sum_{z\in S_2(x)}p_{yz}+3p_{xy}\sum_{y'\in S_1(x)}p_{yy'}+\sum_{y'\in S_1(x)}p_{xy'}p_{y'y}\right]f(y)g(y)\\
&-\sum_{y\in S_1(x)}\sum_{y'\in S_1(x)}2p_{xy}p_{yy'}(f(y')g(y)+f(y)g(y'))\\
&-\sum_{y\in S_1(x)}\sum_{z\in S_2(x)}2p_{xy}p_{yz}(f(y)g(z)+g(y)f(z))+\sum_{z\in S_2(x)}p_{xz}^{(2)}f(z)g(z).
\end{align*}
In the above, we use the notation
\[p_{xz}^{(2)}:=\sum_{y\in S_1(x)}p_{xy}p_{yz}.\]
Therefore we have the expression of the matrix $I(x)_{\hat{1}}$ as below.
For any $y\in S_1(x)$,
\begin{equation*}
I(x)_{yy}=3p_{xy}p_{yx}+3p_{xy}\sum_{z\in S_2(x)}p_{yz}+3p_{xy}\sum_{y'\in S_1(x)}p_{yy'}+\sum_{y'\in S_1(x)}p_{xy'}p_{y'y}.
\end{equation*}
For any $y_i,y_j\in S_1(x)$ such that $y_i\neq y_j$,
\begin{equation*}
I(x)_{y_iy_j}=-2p_{xy_i}p_{y_iy_j}-2p_{xy_j}p_{y_jy_i}.
\end{equation*}
For any $y\in S_1(x)$ and $z\in S_2(x)$,
\begin{equation*}
I(x)_{yz}=I(x)_{zy}=-2p_{xy}p_{yz}, \,\,I(x)_{zz}=p_{xz}^{(2)}.
\end{equation*}
For any $z_i,z_j\in S_2(x)$ such that $z_i\neq z_j$,
\begin{equation*}
I(x)_{z_iz_j}=0.
\end{equation*}
Combining the above expressions for $I(x)_{\hat{1}}$ with (\ref{eq:DeltaDeltaT}),(\ref{eq:Gamma=diag}) and (\ref{eq:Gamma2GammaDelta}) yields the expressions for $\Gamma_2(x)_{\hat{1}}$ as below.

For any $y\in S_1(x)$,
\begin{align} \label{eq:Gamma_2_yy}
(4\Gamma_2(x))_{yy} = &2p_{xy}^{2}+3p_{xy}p_{yx}-\frac{d_x}{\mu_x}p_{xy} +3p_{xy}\sum_{z\in S_{2}(x)}p_{yz} \\
& +\sum_{y'\in S_1(x)}\left( 3p_{xy}p_{yy'} + p_{xy'}p_{y'y} \right). \nonumber
\end{align}
For any $y_i,y_j\in S_1(x)$ such that $y_i\neq y_j$,
\begin{align} \label{eq:Gamma_2_yy_off}
(4\Gamma_2(x))_{y_iy_j}=2 p_{xy_i}p_{xy_j} -2p_{xy_i}p_{y_iy_j} -2p_{xy_j}p_{y_jy_i},
\end{align}
and, for any $z\in S_2(x)$,
\begin{equation} \label{eq:Gamma_2_yz}
(4\Gamma_2(x))_{yz} = (4\Gamma_2(x))_{zy}=-2p_{xy}p_{yz};
\end{equation}
\begin{equation} \label{eq:Gamma_2_zz}
(4\Gamma_2(x))_{zz} = p^{(2)}_{xz},
\end{equation}
and for any $z_i,z_j\in S_2(x)$ such that $z_i\neq z_j$,
\begin{equation} \label{eq:Gamma_2_zz_off}
(4\Gamma_2(x))_{z_iz_j}=0.
\end{equation}

The Schur's complement $$Q(x):=\Gamma_2(x)_{\hat{1}}/ \Gamma_2(x)_{S_2,S_2} =\Gamma_2(x)_{S_1,S_1}-\Gamma_2(x)_{S_1,S_2} \Gamma_2(x)_{S_2,S_2}^{-1}\Gamma_2(x)_{S_2,S_1}$$ is the result of folding the matrix in \eqref{eq:Gamma_2_1hat} into the upper-left block.

For all $y_i,y_j\in S_1(x)$ (with possibly $i=j$), the $(y_i,y_j)$-entry of $\Gamma_2(x)_{S_1,S_2}\Gamma_2(x)_{S_2,S_2}^{-1}\Gamma_2(x)_{S_2,S_1}$ can be computed from \eqref{eq:Gamma_2_yz}, \eqref{eq:Gamma_2_zz} and \eqref{eq:Gamma_2_zz_off} as
\begin{align}\label{eq:Schur2}
\left(\Gamma_2(x)_{S_1,S_2}\Gamma_2(x)_{S_2,S_2}^{-1}\Gamma_2(x)_{S_2,S_1}\right)_{y_i y_j}
= \sum_{z\in S_2(x)} \frac{p_{xy_i}p_{y_i z} \cdot p_{xy_j}p_{y_j z}}{p^{(2)}_{xz}}.
\end{align}

Combining the above equation with \eqref{eq:Gamma_2_yy} and \eqref{eq:Gamma_2_yy_off}, we obtain the entries of $Q(x)$ as follows.

For $y\in S_1(x)$,
\begin{eqnarray} \label{eq:Q_ondiag}
	Q(x)_{yy} &=& \frac{1}{2}p_{xy}^{2}+\frac{3}{4}p_{xy}p_{yx} - \frac{1}{4}\frac{d_x}{\mu_x}p_{xy} +\frac{3}{4}p_{xy}\sum_{z\in S_{2}(x)}p_{yz} \\
	&&+\frac{1}{4}\sum_{y'\in S_1(x)}\left( 3p_{xy}p_{yy'} + p_{xy'}p_{y'y} \right) - \sum_{z\in S_2(x)} \frac{p_{xy}^2p_{y z}^2}{p^{(2)}_{xz}} \nonumber.
\end{eqnarray}

For $y_i,y_j\in S_1(x)$ such that $y_i\neq y_j$,
\begin{eqnarray} \label{eq:Q_offdiag}
	Q(x)_{y_i y_j} &=& \frac{1}{2} p_{xy_i}p_{xy_j} -\frac{1}{2}p_{xy_i}p_{y_iy_j} -\frac{1}{2}p_{xy_j}p_{y_jy_i} - \sum_{z\in S_2(x)} \frac{p_{xy_i}p_{y_i z}p_{xy_j}p_{y_j z}}{p^{(2)}_{xz}}.
\end{eqnarray}

The curvature matrix $A_\infty(x)=2\diag(\mathbf{v}_0)^{-1}Q(x)\diag(\mathbf{v}_0)^{-1}$ with \[\mathbf{v}_0(x) := (\sqrt{p_{xy_1}} \ \sqrt{p_{xy_2}} \ ... \ \sqrt{p_{xy_m}})^\top\] has its entries equal to
\begin{equation} \label{eq:A_infty_entries}
A_\infty(x)_{y_iy_j}= \frac{2}{\sqrt{p_{xy_i}p_{xy_j}}} Q(x)_{y_i y_j}
\end{equation}
for all $y_i,y_j\in S_1(x)$ (with possibly $i=j$).

The generalised scalar curvature $S_{G,x}(N)=\trace (A_\infty - \frac{2}{N}\v_0\v_0^\top)=S_{G,x}(\infty)-\frac{2}{N}\frac{d_x}{\mu_x}$ can then be computed from \eqref{eq:Q_ondiag} as below:
\begin{align} \label{eq:scalar_compute}
S_{G,x}(\infty)
%= \sum_{y\in S_1(x)} \frac{2}{p_{xy}}Q(x)_{yy}  \nonumber \\
&= \left(1-\frac{m}{2}\right)\frac{d_x}{\mu_x} + \frac{3}{2}\sum_{y\in S_1(x)}p_{yx} + \frac{3}{2}\sum_{y\in S_1(x)}\sum_{z\in S_2(x)}p_{yz} \nonumber \\
&\phantom{=} + \frac{1}{2}\sum_{y\in S_1(x)}\sum_{y'\in S_1(x)}\left(3p_{yy'}+\frac{p_{xy'}p_{y'y}}{p_{xy}}\right) - 2\sum_{y\in S_1(x)}\sum_{z\in S_2(x)} p_{xy}\frac{p_{yz}^2}{p_{xz}^{(2)}}.
\end{align}

Next we analyse the structure of the matrix $Q(x)$ via certain Laplacians, extending results in \cite[Section 8]{CLP20}. Let $\Delta_{S_1(x)}$ be the Laplacian of the weighted graph with the vertex set $\{y_1,y_2,\ldots,y_m\}$, the vertex measure $\mu\equiv 1$, and the symmetric edge-weight function given by
\begin{equation*}
w^{S_1(x)}_{y_iy_j}:=\frac{1}{2}p_{xy_i}p_{y_iy_j}+\frac{1}{2}p_{xy_j}p_{y_jy_i}.
\end{equation*}
That is, for any function $f:\{y_1,y_2,\ldots,y_m\}\to \mathbb{R}$, we have
\begin{equation*}
\Delta_{S_1(x)}f(y_i)=\sum_{j\in [m]}w^{S_1(x)}_{y_iy_j}(f(y_j)-f(y_i)),
\end{equation*}
where we use the notation $[m]:=\{1,2,\ldots,m\}$. We observe that
\begin{equation*}\label{eq:DegS1}
\sum_{j\in [m]}w^{S_1(x)}_{y_iy_j}=\frac{1}{2}p_{xy_i}\sum_{y_j\in S_1(x)}p_{y_iy_j}+\frac{1}{2}\sum_{y_j\in S_1(x)}p_{xy_j}p_{y_jy_i}.
\end{equation*}
We then derive from \eqref{eq:Gamma_2_yy} and \eqref{eq:Gamma_2_yy_off} that
\begin{equation}\label{eq:Gamma2S1S1Laplacian}
\Gamma_2(x)_{S_1,S_1}=-\Delta_{S_1(x)}+\frac{1}{2}\left(\Delta(x)\Delta(x)^\top\right)_{\hat{1}}-\frac{1}{4}\frac{d_x}{\mu_x}\diag(\Delta(x)_{S_1}) +\diag(\mathbf{w}_1(x)),
\end{equation}
where $\Delta_{S_1(x)}$ stands here for the corresponding Laplacian matrix and $\mathbf{w}_1(x)$ denotes the $m$-dimensional vector with the $i$-th entry given by
\begin{equation*}
\frac{3}{4}p_{xy_i}(p^-(y_i)+p^+(y_i))+\frac{1}{4}\sum_{y'\in S_1(x)}(p_{xy_i}p_{y_iy'}-p_{xy'}p_{y'y_i}).
\end{equation*}
In the above we use the notations $p^-(y)=p_{yx}$ and $p^+(y)=\sum_{z\in S_2(x)}p_{yz}$ for $y\in S_1(x)$.

Let $\Delta_{S_1'(x)}$ be the Laplacian on the weighted graph with the vertex set $\{y_1,y_2,\ldots,y_m\}$, the vertex measure $\mu\equiv 1$, and the symmetric edge-weight function given by
\begin{equation*}
w^{S_1'(x)}_{y_iy_j}:=\sum_{z\in S_2(x)}\frac{p_{xy_i}p_{y_iz}p_{xy_j}p_{y_jz}}{p_{xz}^{(2)}}\,\,\text{for }\,\,i\neq j,\,\,\text{and }\,\,0\,\,\text{otherwise}.
\end{equation*}
As an operator, we have for any function $f:\{y_1,y_2,\ldots,y_m\}\to \mathbb{R}$,
\begin{equation*}
\Delta_{S_1'(x)}f(y_i)=\sum_{j\in [m]}w^{S_1'(x)}_{y_iy_j}(f(y_j)-f(y_i)).
\end{equation*}
Observe that
\begin{equation*}
\sum_{j\in [m]}w^{S_1'(x)}_{y_iy_j}=p_{xy_i}p^+(y_i)-\sum_{z\in S_2(x)}\frac{p_{xy_i}^2p_{y_iz}^2}{p_{xz}^{(2)}}.
\end{equation*}
We then derive form \eqref{eq:Schur2} that
\begin{equation}\label{eq:Schur2Laplacian}
\Gamma_2(x)_{S_1,S_2}\Gamma_2(x)^{-1}_{S_2,S_2}\Gamma_2(x)_{S_2,S_1}=\Delta_{S_1'(x)}+\diag((p_{xy_1}p^+(y_1)\,\cdots \,p_{xy_m}p^+(y_m))^\top).
\end{equation}
Combing \eqref{eq:Gamma2S1S1Laplacian} and \eqref{eq:Schur2Laplacian}, we arrive at
\begin{equation}\label{eq:QLaplacian}
Q(x)=-\Delta_{S_1''(x)}+\frac{1}{2}\left(\Delta(x)\Delta(x)^\top\right)_{\hat{1}} -  \frac{1}{4}\frac{d_x}{\mu_x}\diag(\Delta(x)_{S_1})
+\diag(\mathbf{w}(x)),
\end{equation}
where $\Delta_{S_1''(x)}:=\Delta_{S_1(x)}+\Delta_{S_1'(x)}$ and $\mathbf{w}(x)$ is the $m$-dimensional vector with the $i$-th entry given by
\begin{equation*}
\frac{3}{4}p_{xy_i}p^-(y_i)-\frac{1}{4}p_{xy_i}p^+(y_i)+\frac{1}{4}\sum_{y'\in S_1(x)}(p_{xy_i}p_{y_iy'}-p_{xy'}p_{y'y_i}).
\end{equation*}

%Finally, the curvature matrix $A_\infty(x)=2\diag(\mathbf{v}_0)^{-1}Q(x)\diag(\mathbf{v}_0)^{-1}$ with \[\mathbf{v}_0(x) := (\sqrt{p_{xy_1}} \ \sqrt{p_{xy_2}} \ ... \ \sqrt{p_{xy_m}})^\top\] has its entries equal to
%\begin{equation} \label{eq:A_infty_entries}
%A_\infty(x)_{y_iy_j}= \frac{2}{\sqrt{p_{xy_i}p_{xy_j}}} Q(x)_{y_i y_j}
%\end{equation}
%for all $y_i,y_j\in S_1(x)$ (with possibly $i=j$).

In terms of the Laplacian $\Delta_{S_1''(x)}$, we have the following identity from \eqref{eq:QLaplacian}
\begin{align}\label{eq:curMatrixLaplacian}
A_{\infty}(x)=&-2\diag(\mathbf{v}_0)^{-1}\Delta_{S_1''(x)}\diag(\mathbf{v}_0)^{-1}+\mathbf{v}_0\mathbf{v}_0^\top-\frac{1}{2}\frac{d_x}{\mu_x}{\rm Id}\\
&+\frac{1}{2}\diag\left[\left(
               \begin{array}{c}
                 3p^-(y_1)-p^+(y_1)+\sum_{y'\in S_1(x)}\frac{p_{xy_1}p_{y_1y'}-p_{xy'}p_{y'y_1}}{p_{xy_1}} \\
                  \vdots   \\
                 3p^-(y_m)-p^+(y_m)+\sum_{y'\in S_1(x)}\frac{p_{xy_m}p_{y_my'}-p_{xy'}p_{y'y_m}}{p_{xy_m}}\\
               \end{array}
             \right)\right].
 \nonumber
\end{align}

We conclude this Appendix with the following Lemma.
\begin{lemma}\label{lem:K0infty}
Let $G=(V,w,\mu)$ be a weighted graph. Then we have for any $x\in V$,
\[ \frac{\mathbf{v}_0(x)^\top A_\infty(x) \mathbf{v}_0(x)}{\mathbf{v}_0(x)^\top \mathbf{v}_0(x)} = \frac{1}{2}\left( \frac{d_x}{\mu_x}+ 3\frac{\mu_x}{d_x} p_{xx}^{(2)} - \frac{\mu_x}{d_x} \sum_{z\in S_2(x)} p_{xz}^{(2)}\right) =:\K^{0}_{\infty}(x). \]
\end{lemma}
\begin{proof}
By \eqref{eq:A_infty_entries}, we obtain
\begin{equation*}
\mathbf{v}_0(x)^\top A_\infty(x) \mathbf{v}_0(x)=\sum_{i,j}\sqrt{p_{xy_i}p_{xy_j}}A_\infty(x)_{y_iy_j}=2\sum_{i,j}Q(x)_{y_iy_j}.
\end{equation*}
We will continue the calculation by applying \eqref{eq:QLaplacian}. We observe the following facts:
\begin{equation*}
\sum_{i,j}\Delta_{S_1''}(x)_{y_iy_j}=0 \,\,\,\text{and}\,\,\,\sum_{i}\sum_{y'\in S_1(x)}(p_{xy_i}p_{y_iy'}-p_{xy'}p_{y'y_i})=0.
\end{equation*}
Furthermore, we derive from \eqref{eq:DeltaDeltaT} that $\sum_{i,j}\left(\left(\Delta(x)\Delta(x)^\top\right)_{\hat{1}}\right)_{y_iy_j}=\left(\frac{d_x}{\mu_x}\right)^2$.
%\begin{equation*}
%\sum_{i,j}\left(\left(\Delta(x)\Delta(x)^\top\right)_{\hat{1}}\right)_{y_iy_j}=\left(\frac{d_x}{\mu_x}\right)^2
%\end{equation*}
Therefore, applying \eqref{eq:QLaplacian} yields that
\begin{align*}
2\sum_{i,j}Q(x)_{y_iy_j}=&\frac{1}{2}\left(\frac{d_x}{\mu_x}\right)^2+\frac{1}{2}\sum_i\left(3p_{xy_i}p^-(y_i)-p_{xy_i}p^+(y_i)\right) \\
=&\frac{1}{2}\left(\frac{d_x}{\mu_x}\right)^2+\frac{1}{2}\left(3p_{xx}^{(2)}-\sum_{z\in S_2(x)}p_{xz}^{(2)}\right).
\end{align*}
Recalling that $\mathbf{v}_0(x)^\top \mathbf{v}_0(x)=\frac{d_x}{\mu_x}$, we finish the proof of this lemma.
\end{proof}

\section*{Acknowledgement}
Shiping Liu is supported by the National Key R and D Program of China 2020YFA0713100 and the National Natural Science Foundation of China (No. 12031017). Supanat Kamtue is supported by the Thai Institute for Promotion of Teaching Science and Technology.

\end{document}